\documentclass[12pt]{amsart}
\usepackage{epsfig}
\usepackage{pdfsync}
\usepackage{manfnt}
\newtheorem{thm}{Theorem}[section]
\newtheorem{cor}[thm]{Corollary}
\newtheorem{lem}[thm]{Lemma}
\newtheorem{prop}[thm]{Proposition}

\theoremstyle{definition}
\newtheorem{defn}[thm]{Definition}
\theoremstyle{remark}
\newtheorem{rem}[thm]{Remark}
\numberwithin{equation}{section}

\def\L2{L^{2}}

\def\M{\mathcal{M}}
\def\E{\mathcal{E}}

\def\D{\mathcal{D}}
\def\G{\mathcal{G}}

\def\R{\Bbb{R}}

\def\m1{^{-1}}

\def\H{\mathcal{H}}
\def\F{\mathcal{F}}

\def\O{\Omega}

\begin{document}

\title[]{Listening to the shape of a drum.}%
\author{Fabio Cipriani, Jean-Luc Sauvageot}%
\address{Dipartimento di Matematica, Politecnico di Milano, piazza Leonardo da Vinci 32, 20133 Milano, Italy.} \email{fabio.cipriani@polimi.it}
\address{Institut de Math\'ematiques, de Jussieu -- Paris Rive Gauche, CNRS-Université Denis Diderot, F-75205 Paris Cedex 13, France}
\email{jean-luc.sauvageot@imj-prg.fr}
\footnote{Work supported by Laboratoire Ypatia des Sciences MathÃ©matiques C.N.R.S. France - Laboratorio Ypatia delle Scienze Matematiche I.N.D.A.M. Italy (LYSM).}
\date{May 24, 2021}


\begin{abstract}
The aim of this work is to link the quasiconformal geometry of a Euclidean domain $U$ to the spectral properties of its Dirichlet integral $\D$, through the algebra of multipliers $\M(H^{1,2}(U))$ of the Sobolev space.\\
In the main result we prove that a homeomorphism $\gamma:U\to\gamma(U)$ between Euclidean domains, giving rise to an algebraic isomorphism $a\mapsto a\circ\gamma$ between $\M(H^{1,2}(\gamma(V)))$ and $\M(H^{1,2}(V))$ for any relatively compact domain $V\subseteq U$ and leaving invariant the corresponding fundamental tones (first non zero eigenvalues) of $\D$
\[
\mu_1(\gamma(V),a)=\mu_1(V,a\circ\gamma)\, ,
\]
is quasiconformal. A companion characterization hold true for bounded distortion maps.
In the converse direction we prove that for $n\ge 3$\\
i) the M\"obius group $G(\R^n)$ acts isometrically on the algebra of multipliers $\M(H^{1,2}_e(\R^n))$ of the extended space $H^{1,2}_e(\R^n)$\\
ii) $(\D,H^{1,2}(\R^n))$ is a closable quadratic form on $L^2(\R^n,\Gamma[a])$ with respect to the energy measure $\Gamma[a]=|\nabla a|^2\, dx$ of any $a\in \M(H^{1,2}_e(\R^n))$\\
iii) for any $\gamma\in G(\R^n)$, the form closure $(\D,\F^a)$ of $(\D,H^{1,2}(\R^n))$ is a Dirichlet form on $L^2(\R^n,\Gamma[a])$, unitarily equivalent to $(\D,\F^{a\circ\gamma})$ on $L^2(\R^n,\Gamma[a\circ\gamma])$.\\
The results are based on connections between fundamental tones and ergodic properties of multipliers: in particular, it is shown that the fundamental tone of
$(\D,\F^a)$ on $L^2(U,\Gamma[a])$ is non vanishing $\mu_1(U,a)>0$ for any fully supported $a\in\M(H^{1,2}(U))$, provided there exists a spectral gap for the usual Laplacian.
\end{abstract}
\vskip-2.5truecm
\maketitle
\section{Introduction and description of the results.}
In a seminal paper \cite{W}, H. Weyl showed that dimension and volume of an Euclidean domain may be traced in the asymptotic distribution of the eigenvalues of its Laplace operators. H. Weyl was motivated by questions and conjectures posed by the physicists, starting with J.W.S. Rayleigh with his famous book "The Theory of Sound" and ending with by H.A. Lorentz and A. Sommerfeld, who were concerned with problems arising in black body radiation theory and in particular by those leading to Planck's law.\\
In a as much famous paper \cite{K}, titled "Can one hear the shape of a drum?", M. Kac popularized this and related mathematical problems connecting geometry and spectrum. He noticed that the hope to characterize {\it isometrically} Euclidean domains or compact Riemannian manifolds by the spectrum of their Laplace operator is vain: J. Milnor had showed in \cite{M} the existence of non isometric 16 dimensional flat tori sharing a common spectrum (see \cite{ANPS} for a careful review of these subjects).\\
The aim of this work is to show that one has indeed some access to "the shape of a drum" in terms of the spectrum of Laplace type operators whose quadratic form is the Dirichlet integral $\D$, provided one first retains for the word {\it shape} the latin sense of {\it forma}: we will thus intend that {\it domains have the same shape if and only if they are conformal, i.e. if and only if they are  transformed one into the other by a map preserving angles}.\footnote{Late Latin: {\it conformalis = having the same shape}}
\vskip0.2truecm\noindent
To discuss the novelties we introduce in the present work to approach the above problem, we recall that in Spectral Geometry, one extracts geometric information about a bounded domain $U$ with smooth boundary, from the list of eigenvalues $\{\lambda_k(U):k\ge 0\}$ of its Laplace operator $-\Delta_N$, subjects, for example, to Neumann conditions on $\partial U$. These spectral data are the critical values on the unit sphere of $L^2(U,dx)$
of the Dirichlet integral $\D$ defined on the Sobolev space $H^{1,2}(U)$ (i.e. the quadratic form of $-\Delta_N$).\\
In this work the quasiconformal class of $U$ will be looked for in different spectral data: in synthesis, {\it we disregard overtones and concentrate on the fundamental ones}. Instead of considering the whole spectrum $\{\lambda_k(U):k\ge 0\}$ determined by $(\D,H^{1,2}(U))$ and the Lebesgue measure $dx$ on $U$, we consider the family of {\it first nonzero eigenvalues or fundamental tones}
$\mu_1(A,a)$ of $\D$ on subdomains $A\subseteq U$, endowed with the energy measures $\Gamma[a]:=|\nabla a|^2\cdot dx$ of multipliers $a\in\M(H^{1,2}(A))$ of the Dirichlet space $H^{1,2}(A)$. In other words, we consider the family of first nonzero critical values of the Dirichlet integral $(\D, H^{1,2}(A))$ on $L^2(A,\Gamma[a])$, parametrized by subdomains of $U$ and by functions in the algebra of multipliers of $H^{1,2}$(A).\\
In ultimate analysis, the quasiconformal class of $U$ will be naturally connected to the potential theory of the energy form $(\D,H^{1,2}(U))$ so that the volume measure $dx$ of $U$ plays no role.
\vskip0.1truecm\noindent
The natural framework is thus that of Dirichlet form theory, i.e. the one of a kernel-free Potential Theory based on the notion of energy. For example, the replacement of the Lebesgue measure $dx$ of $U$ by the energy measures $\Gamma[a]$ of multipliers $a$ of $H^{1,2}(U)$ is managed by the boundary theory of Dirichlet forms (\cite{CES}) and, in particular, by the process of {\it change of speed measure} and by the one of taking the {\it trace of $\D$ with respect to smooth measures}. In these processes a decisive role is played by the extended Dirichlet space $H^{1,2}_e(U)$.
\vskip0.1truecm\noindent
In connection with the quasiconformal geometry of Euclidean domains $U\subseteq\R^n$ in dimension $n\ge 3$, the multipliers algebra $\M(H^{1,2}(U))$ plays a role alternative to the one played by the Royden algebra $H^{1,n}(U)\cap L^\infty(U,dx)$ (see [Lew], [Mos]). Recall that quasiconformal transformations between Euclidean domains $\gamma:U \to V$ can be characterized algebraically as those which establishes an bounded, invertible endomorphism $a\mapsto a\circ \gamma$ between the, naturally normed, algebras $H^{1,n}(V)\cap L^\infty(V,dx)$ and $H^{1,n}(U)\cap L^\infty(U,dx)$. In this respect, for example, we will see below that the M\"obius group $G(\R^n)$ of a Euclidean space $\R^n$ having dimension $n\ge 3$, naturally acts by {\it isometries} on the algebra $\M(H^{1,2}_e(\R^n))$ of multipliers of the extended Dirichlet space.
\vskip0.1truecm\noindent
An essential difference between the algebras $\M(H^{1,2}(U))$ and $H^{1,n}(U)\cap L^\infty(U,dx)$ relies on the fact that the definition of the former does not involves, explicitly, higher order integrability of the gradient of functions, as the Sobolev space $H^{1,n}(U)$ does, nor the dimension of the Euclidean space. The multiplier algebra $\M(H^{1,2}(U))$ is intrinsic to the Dirichlet space $H^{1,2}(U)$ and reflects aspects of the potential theory of it. Its use makes explicit that the conformal geometry of a Euclidean domain underlies, its energy functional only, with no reference to its volume measure.
\vskip0.1truecm\noindent
The way by which spectral properties of multipliers enter naturally into the play lies in the fact that if $a$ is a multiplier of $H^{1,2}(U)$, i.e. it is a measurable function transforming by pointwise multiplication any $b\in H^{1,2}(U)$ into another element $ab\in H^{1,2}(U)$, and if its energy measure $\Gamma[a]$ has full support, then the Sobolev space $H^{1,2}(U)$ is naturally embedded into $L^2(U,\Gamma[a])$ and the Dirichlet integral $(\D,H^{1,2}(U))$ is closable on $L^2(U,\Gamma[a])$. One can then consider the fundamental tone $\mu_1(U,a)$ of its quadratic form closure $(\D,\F^a)$ on $L^2(U,\Gamma[a])$.
\vskip0.2truecm\noindent
In Section 2 we recall the definition of the multiplier algebra $\M(H^{1,2}(U))$ of the Sobolev space and its natural norm and seminorm. Then we define the fundamental tone $\mu_1(U,a)$ of a multiplier showing the connections of this spectral feature with ergodic properties of the Dirichlet space $H^{1,2}(U)$ such as transience and recurrence and with the spectral gap of the Neumann Laplacian of $U$. Then we bound from below $\mu_1(U,a)$ in terms of the seminorm of the multiplier $a$ and the uniform norm of the potential of the energy measure $\Gamma[a]$.
\vskip0.1truecm\noindent
In Section 3 we recall the definition of maps with bounded distortion, quasiconformal and conformal maps on Euclidean domains, as well as the definition of the M\"obius group $G(\R^n)$ of $\R^n $ and the conformal group $G(\mathbb{S}^n)$ of the standard unit sphere $\mathbb{S}^n$.
\vskip0.1truecm\noindent
In Section 4, using also the known conformal invariance of the Hardy-Littlewood-Sobolev functional, we show on $\R^n$ the {\it conformal covariance} of the Green operator and of the Dirichlet integral $\D$. Later we prove that the M\"obius group $G(\R^n)$ naturally acts as a group of {\it isometries of the multipliers algebra} $\M(H^{1,2}_e(\R^n))$ of the extended Dirichlet space $H^{1,2}_e(\R^n)$ and  that the flow of the energy measures $\Gamma[a]$ of multipliers $a\in\M(H^{1,2}_e(\R^n))$, determined by the action of the M\"obius group, determines a unitary flow among the spaces $L^2(U,\Gamma[a])$.\\
This will enable us to prove that if $\gamma\in G(\R^n)$ is a M\"obius transformation and $a\in\M(H^{1,2}_e(\R^n))$ is a nowhere constant multiplier, then the Dirichlet integral, defined on suitable natural domain $(\D,\F^a)$, is a Dirichlet form on $L^2(\R^n,\Gamma[a])$ and that it is unitarily isomorphic to the Dirichlet form $(\D,\F^{a\circ\gamma})$ on $L^2(\R^n,\Gamma[a\circ\gamma])$, determined by the transformed multiplier $a\circ\gamma\in\M(H^{1,2}_e(\R^n))$.\\
The section ends with two localized versions of this result. The first concerns the part $(\D,(\F^a)_U)$ of the Dirichlet integral on the space $L^2(U,\Gamma[a])$ of a domains $U\subset\R^n$ and with a nowhere constant multipliers $a\in\M(H^{1,2}_e(\R^n))$. The second deals with the trace of the Dirichlet integral with respect to the energy measure of any multiplier $a\in\M(H^{1,2}_e(\R^n))$.
\vskip0.1truecm\noindent
To proceed toward our main result, we recall in Section 5 the notion of {\it conformal volume} $V_c(M)$ of compact manifold $M$, introduced by Li-Yau, and a result due to Colbois-El Soufi-Savo showing how to bound the fundamental tone $\mu_1(M,\nu)$ of the manifold $M$ endowed with a weighted Riemannian measure $\nu$, in terms of its Riemannian volume $V(M)$, its conformal volume $V_c(M)$ and the total weight $\nu(M)$.
\vskip0.1truecm\noindent
Section 6 contains our main result showing that an homeomorphism $\gamma:U\to\gamma(U)$ which gives rise to an algebraic isomorphism $a\mapsto a\circ\gamma$ between
the algebras of finite energy multipliers $\F\M(H^{1,2}(\gamma(A)))$ and $\F\M(H^{1,2}(A))$ of any relatively compact domain $A\subseteq U$ and leaves invariant the corresponding fundamental tones $\mu_1(\gamma(A),a)=\mu_1(A,a\circ\gamma)$, is quasiconformal. We also provide a version of the result dealing with {\it smooth transformations} $\gamma$: in this case the hypotheses about the global invariance of the algebra of finite energy multipliers is automatically satisfied. Companion results are proved for bounded distortion maps.
\vskip0.1truecm\noindent
Section 2 Multipliers of Dirichlet integrals and their fundamental tones.\\
Section 3 Bounded distortion maps, quasiconformal and conformal maps.\\
Section 4 M\"obius transformations as automorphims of a multipliers algebra.\\
Section 5 Conformal volume and fundamental tone of multipliers.\\
Section 6 Bounded distortion of fundamental tones and bounded distortion of maps.\\
Section 7 Conclusions.
\vskip0.1truecm\noindent
{\bf Acknowledgments}. The authors wish to thanks Alessandro Savo, Universita' "La Sapienza" Roma,  for several discussions about the subjects touched by the present work.
\section{Multipliers of Dirichlet integrals and their fundamental tones}

We recall in this section, the definition of multipliers of Sobolev spaces and Dirichlet integrals on Euclidean domains. For the background material we refer to the monograph [MS]. Multipliers of general Dirichlet spaces are studied in [CS]. For the needed aspects of the theory of Dirichlet forms on locally compact spaces we refer to [CF].
\vskip0.2truecm
\subsection{Sobolev spaces}
Let $U\subseteq\mathbb{R}^n$ be a (connected) domain in an $n$-dimensional Euclidean space, endowed with its Lebesgue measure $dx$. For any real $p\ge 1$ and integer $m\ge 1$, we will denote by $L^p(U,\mathbb{R}^m,dx)$, or simply by $L^p(U,\mathbb{R}^m)$, the usual space of functions whose components are $p$-integrable on $U$.
\par\noindent
The Sobolev space $H^{1,p}(U,\mathbb{R}^m)$ is the set of functions $\gamma\in L^p(U,\mathbb{R}^m)$ which are "absolutely continuous along lines" $\gamma\in {\rm ACL\,}(U,\mathbb{R}^m)$ and such that their first order derivative $\gamma^\prime=[\partial_j \gamma_i]_{i,j=1}^{n,m}$, which is then defined a.e. in $U$, are $p$-integrable $\gamma^\prime\in L^p(U,\mathbb{R}^{m\times n})$.
\par\noindent
The Sobolev space $H^{1,p}_{\rm loc}(U,\mathbb{R}^m)$, $p\ge 1$, is the set of functions in $L^p_{\rm loc}(U,\mathbb{R}^m)$
which are "absolutely continuous along lines" $\gamma\in {\rm ACL\,}(U,\mathbb{R}^m)$ and such that their derivatives are locally $p$-integrable $\gamma^\prime\in L^p_{\rm loc}(U,\mathbb{R}^{m\times n})$.
\par\noindent
Let us  recall also the definition of the space of Beppo Levi functions (see e.g. [CF, 2.2.4])
\[
{\rm BL}(U):=\{b\in {\rm ACL}(U): \nabla b\in L^2(U,\R^n)\}
\]
and its quotient space by the subspace of constant functions ${\overset{\circ}{{\rm BL}}(U)}$. The seminorm $a\mapsto \sqrt{\D[a]}$ on ${\rm BL}(U)$, is a complete norm on ${\overset{\circ}{{\rm BL}}(U)}$ and the following identification holds true
\[
H^{1,2}(U)={\rm BL}(U)\cap L^2(U,dx)\, .
\]
The space of Beppo Levi functions can also be identified with a space of Schwartz distributions
\[
{\rm BL}(U)=\{b\in \mathcal{D}'(U):\nabla b\in L^2(U,\R^n)\}
\]
where the derivatives are taken in the distribution sense. Distributions in ${\rm BL}(U)$ can be identified with functions in $L^2_{\rm loc}(U,dx)$ so that we have too
\[
{\rm BL}(U)=\{b\in L^2_{\rm loc}(U,dx):\nabla b\in L^2(U,\R^n)\}\, .
\]
\subsection{Dirichlet integrals and their energy measures (carr\'e du champ)}
A distinguished role is played, in this work, by the Sobolev space $H^{1,2}(U)$ as form domain of the Dirichlet integral $(\D,H^{1,2}(U))$
\begin{equation}\label{DF}
  \D[b]:=\int_U |\nabla b(x)|^2\cdot dx\qquad b\in H^{1,2}(U)\, ,
\end{equation}
a closed, symmetric, real, quadratic form on the Hilbert space $L^2(U,dx)$, whose associated self-adjoint, nonnegative operator
$(L,D(L))$ is an extension of the Laplace operator $(-\Delta,C^\infty_c(U))$. It possesses the characteristic contraction property, called Markovianity,
\[
\D[b\wedge 1]\le\D[b]\qquad b=\bar{b}\in H^{1,2}(U)
\]
which makes the semigroup $\{e^{-tL} :t\ge 0\}$ Markovian, in the sense that it is positivity preserving and contractive on each Lebesgue spaces $L^p(U,dx)$, for any $p\in [1,+\infty]$.
\par\noindent
The Sobolev space $H^{1,2}(U)$, when  considered as a Hilbert space under the {\it graph norm}
\[
\|b\|_{H^{1,2}(U)}:=\sqrt{\D[b]+\|b\|^2_{L^2(U)}}\qquad b\in H^{1,2}(U)\, ,
\]
will be called the {\it Dirichlet space} relative to the Dirichlet integral $(\D,H^{1,2}(U))$ on $L^2(U,dx)$. Any locally finite-energy function $a\in H^{1,2}_{\rm loc}(U)$ determines a positive Radon measure on $U$
\[
\Gamma[a]:=|\nabla a|^2\cdot dx\, ,
\]
called its {\it energy measure} or the {\it carr\'e du champ}.
\vskip0.2truecm\noindent
The Dirichlet integral is the archetypical example of a {\it Dirichlet form}, whose general definition we now briefly recall. Given a locally compact, metrizable Hausdorff space, endowed with a positive Radon measure of full topological support $(X,m)$, a {\it Dirichlet form} $(\E,\F)$ with respect to $L^2(X,m)$ is a densely defined, real, lower semicontinuous quadratic form $\E:\F\to [0,+\infty)$ which is Markovian in the sense that
\[
\E[b\wedge 1]\le\E[b]\qquad b={\bar b}\in\F\, .
\]
The form domain $\F$, when endowed with the graph norm $\|b\|_\F :=\sqrt{\E[b]+\|b\|^2_{L^2(U,m)}}$, is a Hilbert space called the {\it Dirichlet space} relative to the Dirichlet form $(\E,\F)$ on $L^2(X,m)$.
\vskip0.1truecm\noindent
The Dirichlet forms considered in this paper will be suitable variations of the Dirichlet integral $(\D,H^{1,2}(U))$ on $L^2(U,dx)$. In several situations, the quadratic form  $\E$ will always be given by the Dirichlet integral $\D$ whereas the reference measures $m$ will be suitable positive Radon measures, absolutely continuous with respect to the Lebesgue's one. Also, beside $H^{1,2}(U)$, different choices of form domains $\F\subset L^2(U,m)$ will be considered.
\subsection{Extended Dirichlet spaces and random time change}
The {\it extended Dirichlet space} $\F_e$ of a Dirichlet space $(\E,\F)$ on $L^2(X,m)$ is the set of $m$-equivalence classes  of measurable functions $b$ on $X$, which are finite $m$-a.e. and for which there exists a $\D$-Cauchy sequence $b_n\in\F$ in the sense that $\D[b_n-b_m]\to $ as $n,m\to +\infty$,  such that $b_n\to b$ pointwise $m$-a.e.. For all approximating sequence, the limit $\E[b]:=\lim_n\E[b_n]$ exists and attains the same value, thus defining an extension of the quadratic form $\E$ to $\F_e$. Moreover, one recovers the Dirichlet space as $\F:=\F_e\cap L^2(X,m)$.
\vskip0.1truecm\noindent
A $m$-measurable set $A\subseteq X$ is said to be {\it invariant} for $(\E,\F)$ if one has
\[
1_A\cdot b\in\F\, ,\qquad \E[b]=\E[1_A\cdot b]+\E[1_{A^c}\cdot b]\qquad b\in\F\, .
\]
The Dirichlet form is said to be {\it irreducible} if $X$ and $\emptyset$ are the only invariant set and it is said to be {\it transient} if $b\in\F_e$ and $\E[b]=0$ implies $b=0$. In the latter situation $\F_e$ is a Hilbert space under the norm $\sqrt{\E[\cdot]}$. A Dirichlet form $(\E,\F)$ on $L^2(X,m)$ is said to be {\it recurrent} if $1\in\F_e$ and $\E[1]=0$. An irreducible Dirichlet space which is not recurrent is necessarily transient and viceversa.
\vskip0.1truecm\noindent
In this paper we are firstly concerned with the Dirichlet space $(\D,H^{1,2}(U))$ on $L^2(U,dx)$. As the open set $U$ is assumed to be connected, the Dirichlet space is irreducible. In general $H^{1,2}_e(U)\subseteq {\rm BL(U)}$ for all domains $U\subseteq \mathbb{R}^n$ and any dimension $n\ge 1$. In particular, when $U\subseteq \mathbb{R}^n$ in dimension $n=1,2$ or when the volume $|U|$ is finite, the Dirichlet space is recurrent and $H^{1,2}_e(U)={\rm BL(U)}$. Moreover, $(\D,H^{1,2}(\mathbb{R}^n))$ on $L^2(\mathbb{R}^n)$ is transient if and only if $n\ge 3$ and, in this case, $H^{1,2}_e(\mathbb{R}^n)\simeq{\overset{\circ}{{\rm BL}}(\mathbb{R}^n)}$.
\vskip0.1truecm\noindent
Dirichlet forms are in one to one correspondence to Markovian semigroups $\{e^{-tL}:t>0\}$ on $L^2(X,m)$ with self-adjoint, nonnegative generator
$(L,D(L))$, through the relation
\[
\E[b]=\|\sqrt{L}b\|^2_{L^2(X,m)}\qquad b\in\F= D(\sqrt L)\, .
\]
When the Dirichlet form is {\it regular} in the sense that $\F\cap C_0(X)$ is dense in the Dirichlet space $\F$ (i.e. it is a form core) and it is uniformly dense in the Banach space $C_0(X)$ of continuous functions vanishing at infinity, one considers a Choquet capacity defined initially as
\[
{\rm Cap}_1(O):=\inf\{\|b\|^2_\F:b\in\F, b\ge {\bf 1}_O,\,\,{\rm m}-a.e.\,\,{\rm on}\,\, O\}
\]
on open sets $O\subseteq X$ and then extended to Borel sets $B\subseteq X$ as
\[
{\rm Cap}_1(B):=\inf\{{\rm Cap}_1(O):O\,\,{\rm open}\,\,,B\subseteq O\}\, .
\]
The set function $\rm Cap_1$ allows to define the class of negligible Borel sets $B$ of the potential theory of Dirichlet forms as those having vanishing capacity ${\rm Cap}_1(B)=0$. Capacity zero sets, also termed {\it $\E$-polar set}, necessarily have vanishing measure $m(B)=0$. Properties valid {\it except}, possibly, on a $\E$-polar set are said to subsist {\it quasi-everiwhere}, abbreviated {\it q.e.}. In the transient case the same class of negligible sets can be described using the Choquet capacity $\rm Cap_0$, defined analogously to $\rm Cap_1$ except for the fact that one minimizes the quadratic form $\E$ instead that the quadratic form $\|\cdot\|^2_\F :=\E+\|\cdot\|^2_2$.\\
To any regular Dirichlet form on $L^2(X,m)$ is associated an essentially unique, $m$-symmetric Hunt stochastic process on $X$. Irreducibility, recurrence and transience of a Dirichlet spaces then have a natural probabilistic dynamical interpretations in terms of the associated process.
\vskip0.1truecm\noindent
When the domain $U\subset\R^n$ has {\it continuous boundary} $\partial U$, in the sense of Maz'ya ([Ma Theorem 2 page 14]), the Dirichlet form $(\D,H^{1,2}(U))$ on $L^2(U\cup\partial U,1_U\cdot dx)\simeq L^2(U,dx)$ is regular and the associated Hunt processes is the Brownian motion in $U\cup\partial U$ {\it reflected at the boundary} $\partial U$ ([CF Ch. 6]).\\
On any open region $U\subseteq\R^n$, if the domain of the Dirichlet integral $\D$ on $L^2(U,dx)$ is restricted to the subspace $H^{1,2}_0(U)$ defined as the closure of $C^\infty_c(U)$ in the Sobolev space $H^{1,2}(U)$, a regular Dirichlet form is obtained whose associated Hunt process is the Brownian motion in $U$ {\it absorbed at the boundary} $\partial U$.
\subsection{Part and trace of a Dirichlet integral, random time change}
Here we briefly recall three fundamental ways to perturb a Dirichlet space $(\E,\F)$ on $L^2(X,m)$ and its associated process, that we will use in the forthcoming sections.\\
If $U\subseteq X$ is a fixed open set and one defines $\F_U:=\{b\in\F: b=0\,\,{\rm q.e.}\,\, {\rm on}\,\, U^c\}$, then $(\E,\F_U)$ is a regular Dirichlet form on $L^2(U,m)$, called the {\it part of the Dirichlet form $(\E,\F)$} on $U$. For example, for any domain $U\subset\R^n$, the Dirichlet integral $(\D,H^{1,2}_0(U))$ on $L^2(U,dx)$ introduced above, can be described as the part on $U$ of the Dirichlet integral $(\D,H^{1,2}(\R^n))$ on $L^2(\R^n,dx)$.\\
If instead, one fixes a positive Radon measure $\nu$ on $X$ having full topological support ${\rm supp}(\nu)=X$ and charging no $\E$-polar sets, then, setting $\check\F:=\F_e\cap L^2(X,\nu)$ one obtains a well defined, regular Dirichlet form $(\E,\check\F)$ on $L^2(X,\nu)$. This measure changing from $m$ to $\nu$ has a clear interpretation at the dynamical level: the process $Y^\nu$ on $X$ associated to the measure $\nu$ differs from the process $Y^m$ associated to the measure $m$ by a {\it random time change} only, in the sense that $Y^\nu_t=Y^m_{\tau_t}$ for a a time process $\tau$ on $\R_+$ depending on $\nu$ and $m$. This is also the reason why the reference measures $m$ and $\nu$ are called {\it speed measures}.\\
In the forthcoming section the change of speed measure will be considered for the Dirichlet integral $(\D,H^{1,2}(U)$ on Euclidean domains $U$ with respect to the Radon energy measures $\Gamma[a]:=|\nabla a|^2\cdot dx$ for $a\in H^{1,2}_{\rm loc}(U)$. This is allowed by the fact that $\D$-polar sets necessarily have vanishing Lebesgue measure so that, a fortiori, they have vanishing $\Gamma[a]$-measure too.
\vskip0.2truecm
The procedure to take {\it trace of a Dirichlet form} on the support ${\rm supp}(\nu)\subseteq X$ of a Radon measure charging no $\E$-polar sets is described in Section 4.4 for the case of Dirichlet integral.

\subsection{Multipliers of Dirichlet integrals}
A central object in this work  is played by the algebra of multipliers $\M(H^{1,2}(U))$ of the Dirichlet integral.
\begin{defn}(Dirichlet integral multipliers)\label{multiplier1}\par\noindent
A multiplier of the Dirichlet integral $(\D,H^{1,2}(U))$ on $L^2(U,dx)$, is a measurable function $a$ on $U$ such that $ab\in H^{1,2}(U)$ for all
$b\in H^{1,2}(U)$.
\end{defn}
\noindent
If $V\subset U$ is a subdomain with the extension property (see e.g. \cite{Ma}), then the restriction to $V$ of a multiplier of $H^{1,2}(U)$, provides a multiplier of $H^{1,2}(V)$.
\vskip0.1truecm\noindent
If $a$ is a multiplier and $b_n, b,b'\in H^{1,2}(U)$ are such that $\|b_n-b\|_{H^{1,2}(U)}\to 0$ and $\|ab_n-b'\|_{H^{1,2}(U)}\to 0$ as $n\to+\infty$, then $\|b_n-b\|_2\to 0$ and $\|ab_n-b'\|_2\to 0$ as $n\to+\infty$. Thus, possibly passing to a suitable subsequence, we have that $b_n\to b$ and $ab_n\to b'$ $dx$-a.e. on $U$ so that $b'=ab$. The function $a$, as an operator on the Hilbert space $H^{1,2}(U)$, is then closed and everywhere defined thus, by the Closed Graph Theorem, is a bounded operators on $H^{1,2}(U)$. \\
Multipliers form a unital algebra, denoted by $\M(H^{1,2}(U))$, which inherits a natural norm
\[
\|a\|_{\M(H^{1,2}(U))} :=\sup\{\|ab\|_{H^{1,2}(U)}:\|b\|_{H^{1,2}(U)}\le 1\}
\]
as a subalgebra of the Banach algebra of all bounded operators on the Hilbert space $H^{1,2}(U)$. A special role will be played by the subalgebra of {\it finite energy multipliers} defined as follows
\[
\F\M(H^{1,2}(U)):=\M(H^{1,2}(U))\cap H^{1,2}(U).
\]
\vskip0.2truecm\noindent
If $a$ is a multiplier and $b\in H^{1,2}(U)$ is a function that nowhere vanishes $dx$-a.e. on $U$, then
\[
\|a\|_\infty=\lim_{n\to+\infty}\|a^n b\|_2^{1/n}\le \lim_{n\to+\infty}\|a^n b\|_{H^{1,2}(U)}^{1/n}\le \lim_{n\to+\infty}\|a^n\|^{1/n}_{\mathbb{B}(H^{1,2}(U))}\cdot \|b\|_{H^{1,2}(U)}^{1/n}=\|a\|_{\M(H^{1,2}(U))}
\]
so that we have a contractive embedding $\M(H^{1,2}(U))\to L^\infty(U,dx)$ of the multiplier algebra into the algebra of essentially bounded functions.
\vskip0.2truecm\noindent
In the particular case where $U=\R^n$ for $n\ge 3$, the Dirichlet space $(\D,H^{1,2}(\R^n))$ on $L^2(\R^n,dx)$ is transient so that the energy seminorm $\sqrt{\D}$ is a complete norm on the extended Dirichlet space $H^{1,2}_e(\R^n)$. In this situations we may and we shall consider the multipliers algebra $\M(H^{1,2}_e(\R^n))$ of the extended Dirichlet space too.
\vskip0.2truecm\noindent
Since any point of $U$ has a neighborhood on which a suitable function in $H^{1,2}(U)$ takes the constant value $1$, it follows that
$\M(H^{1,2}(U))\subset L^\infty(U,dx)\cap H^{1,2}_{\rm loc}(U)$. On the other hand, the Leibnitz rule
\begin{equation}\label{leibnitz}
\nabla (ab)=(\nabla a)b+a\nabla b\qquad a\in H^{1,2}_{\rm loc}(U)\, ,\quad b\in H^{1,2}(U)\, ,
\end{equation}
interpreted as an identity in $L^1_{\rm loc}(U,\mathbb{R}^n)$, implies that a function $a\in L^\infty(U,dx)\cap H^{1,2}_{\rm loc}(U)$  is a multiplier if and only if the map
\[
T_a:H^{1,2}(U)\to L^2(U,\mathbb{R}^n)\qquad T_a(b):=(\nabla a)b
\]
is a bounded operator. The functional
\[
\eta :\M(H^{1,2}(U))\to [0,+\infty)\qquad \eta(a):=\|T_a\|_{H^{1,2}(U)\to L^2(U,\mathbb{R}^n)}
\]
is a seminorm on $\M(H^{1,2}(U))$ (vanishing on the constant functions only) which will be referred to as the {\it multipliers seminorm}. Since for $a\in L^\infty(U,dx)\cap H^{1,2}_{\rm loc}(U)$ we have
\[
\|b\|_{L^2(U,\Gamma[a])}=\|(\nabla a)b\|_{L^2(U,\mathbb{R}^n)} \qquad b\in H^{1,2}(U)\, ,
\]
the Leibnitz rule implies that $a$ is a multiplier if and only if $H^{1,2}(U)\subset L^2(U,\Gamma[a])$ and, in this case, the norm of the embedding
\[
i_a:H^{1,2}(U)\to L^2(U,\Gamma[a])\qquad i_a(b):=b
\]
coincides with the value of the seminorm $\eta(a)=\|i_a\|_{H^{1,2}(U)\to L^2(U,\Gamma[a])}$. An equivalent norm on $\M(H^{1,2}(U))$ (see [MS Theorem 2.3.1]) is given by $\|a\|_{\M(H^{1,2}(U))} \simeq\eta(a)+\|a\|_\infty$.\\
\vskip0.1truecm\noindent
Notice that, by the Leibnitz rule, a bounded Lipschitz function $a\in {\rm Lip}(U)\cap C_b(U)$ reduces, by restriction, to a multiplier of $H^{1,2}(V)$, for any bounded subdomain $V\subset U$.
\vskip0.1truecm\noindent
The domain $U$ has finite measure if and only if $1\in H^{1,2}(U)$. In this case any multiplier has finite energy
\[
\M(H^{1,2}(U))\subseteq H^{1,2}(U)
\]
and the multiplier seminorm is not weaker than (a multiple of) the energy seminorm
\[
\Bigl(\frac{1}{|U|}\D[a]\Bigr)^{\frac{1}{2}}=\Bigl(\frac{1}{|U|}\int_U|\nabla a|^2\, dx\Bigr)^{\frac{1}{2}}\le \eta(a)\qquad  a\in\M(H^{1,2}(U))\, .
\]
%
Since the energy measure  $\Gamma[a]:=|\nabla a|^2\cdot dx$ of any multiplier $a\in\M(H^{1,2}(U))$, is a positive Radon measure charging no $\D$-polar sets, in case $\Gamma[a]$ has full topological support we may consider the Dirichlet integral $(\D, H^{1,2}(U))$ on the space $L^2(U,\Gamma[a])$. Next proposition collects the first basic relations involved in this process.
\begin{prop}\label{DF}
Let $a\in \M(H^{1,2} (U))$ be a multiplier of the Dirichlet space $(\D,H^{1,2}(U))$ on $L^2(U,dx)$, let $\Gamma[a]$ be its energy measure and assume
${\rm supp}(\Gamma[a])=U$. Then, setting
\[
\F^a:=H^{1,2}_e(U)\cap L^2(U,\Gamma[a]),\qquad \G^a:={\rm BL}(U)\cap L^2(U,\Gamma[a])\, ,
\]
we have
\vskip0.2truecm\noindent
i) $H^{1,2}(U)\subseteq \F^a\subseteq \G^a\subset L^2(U,\Gamma[a])$;
\vskip0.2truecm\noindent
ii) $H^{1,2}(U)$ is dense in $L^2(U,\Gamma[a])$;
\vskip0.2truecm\noindent
iii) $(\D,\G^a)$ is a Dirichlet form on $L^2(U,\Gamma[a])$;
\vskip0.2truecm\noindent
iv) $(\D,\F^a)$ is a Dirichlet form on $L^2(U,\Gamma[a])$;
\vskip0.2truecm\noindent
v) $(D,H^{1,2}(U))$ is a closable Markovian form on $L^2(U,\Gamma[a])$;
\vskip0.2truecm\noindent
vi) $(\D,\F^a)$ is the form closure of $(\D,H^{1,2}(U))$ on $L^2(U,\Gamma[a])$.
 \end{prop}
\begin{proof}
i) By the definition of $\F^a$ and $\G^a$ and the general fact that $H^{1,2}_e(U)\subseteq {\rm BL}(U)$, it is enough to show that $H^{1,2}(U)\subset L^2(U,\Gamma[a])$. To this end, just notice that the Leibnitz rule \eqref{leibnitz}, applied to $a\in \M(H^{1,2} (U))$ and $b\in H^{1,2}(U)$, implies $(\nabla a)b=\nabla(ab)-a\nabla b\in L^2(U,\mathbb{R}^n)$ so that $b\in L^2(U,\Gamma[a])$ because
\[
\int_U |b|^2d\Gamma[a]=\|(\nabla a)b\|^2_{L^2(U,\mathbb{R}^n)}<+\infty\, .
\]
ii) The algebra ${\rm Lip}_c(U)$ of compactly supported Lipschitz functions is contained in $H^{1,2}(U)$ and separates the points of $U$, it is thus uniformly dense in the algebra $C_0(U)$ of continuous functions of compact support in $U$. Since multipliers are in $H^{1,2}_{\rm loc}(U)$, the energy measure $\Gamma[a]$ is a Radon measure so that ${\rm Lip}_c(U)$ is, a fortiori, dense in $L^2(U,\Gamma{a})$.
\vskip0.2truecm\noindent
iii) By item i) we have $H^{1,2}(U)\subseteq \G^a$ so that the quadratic form $(\D,\G^a)$ is densely defined, by item ii). To prove that it is closed, we proceed as follow.
Let $b_n\in \G^a$ be a sequence converging to $b\in L^2(U,\Gamma[a])$ in the norm of $L^2(U,\Gamma[a])$ and such that $\D[b_n-b_m]\to 0$ as $n,m\to +\infty$. We have to show that $b\in \G^a$ (which reduces to show that $b\in {\rm BL}(U)$) and that $\D[b_n-b]\to 0$ as $n\to+\infty$. By the properties of the space of Beppo Levi functions (see e.g. \cite[Section 2.2.4]{CF}), there exists $b' \in {\rm BL}(U)$ and constants $c_n\in\mathbb{C}$ such that $\D[b_n-b^\prime]\to 0$ and $b_n+c_n\to b^\prime$ in $L^2_{\rm loc}(U,dx)$. Then there exists a subsequence such that $b_{n_k}+c_{n_k}\to b^\prime$ as $k\to +\infty$, $dx$-a.e. in $U$. Possibly considering a further subsequence, we have also that $b_{n_k}\to b$ as $k\to +\infty$, $\Gamma[a]$-a.e. in $U$. Then $c_{n_k}=(b_{n_k}+c_{n_k})-b_{n_k}\to b^\prime -b$ as $k\to +\infty$, $\Gamma[a]$-a.e. in $U$ and there exists $c:=\lim_{k\to +\infty} c_{n_k}\in\mathbb{C}$ such that $b=b^\prime -c$, $\Gamma[a]$-a.e. in $U$. Hence $b\in {\rm  BL}(U)$ and consequently $b\in {\rm  BL}(U)\cap L^2(U,\Gamma[a])=\G^a$. Moreover, as $\D[c]=0$, we have $\D[b_n-b]=\D[b_n -(b+c)]=\D[b_n -b^\prime]\to 0$ as $n\to +\infty$ which proves that $(\D,\G^a)$ is closed in $L^2(U,\Gamma[a])$. As the Markovianity property of the Dirichlet integral, $\D[b\wedge 1]\le \D[b]$, is exactly the same when considered on $L^2(U,dx)$ and on $L^2(U,\Gamma[a])$, we have that $(\D,\G^a)$ is a Dirichlet form.
\par\noindent
iv) Let $b_n\in\F^a$ be a sequence such that $\D[b_n-b_m]\to 0$ as $n,m\to +\infty$ and for which there exists $b\in L^2(U,\Gamma[a])$ such that $\|b_n-b\|_{L^2(U,\Gamma[a])}\to 0$ as $n\to +\infty$. Passing to a suitable subsequence if needed, we have that $b_n(x)\to b(x)$ $\Gamma[a]$-a.e. $x\in U$. By the assumption ${\rm supp}(\Gamma[a])=U$, we have that $b_n(x)\to b(x)$ $dx$-a.e. $x\in U$ so that $b\in H^{1,2}_e(U)$ and, a fortiori, $b\in\F^a$.
\par\noindent
v) $(D,H^{1,2}(U))$ is a closable form on $L^2(U,\Gamma[a])$ as restriction of the closed form $(\D,\F^a)$.
\par\noindent
vi) By \cite[Theorem 5.2.8 i)]{CF}, any special standard core of $(\D,H^{1,2}(U))$ on $L^2(U,dx)$  (as ${\rm Lip}_c(U)$, for example) is dense in $\F^a$ with respect to the graph norm of the Dirichlet form $(\D,\F^a)$  on $L^2(U,\Gamma[a])$.

\end{proof}
\begin{rem}
i) The space $\F^a$ is a {\it time changed Dirichlet space} according to \cite[Chapter 5]{CF}. When $U$ has continuous boundary in the sense of Maz'ya, the Dirichlet space $(\D, H^{1,2}(U))$ on $L^2(U\cup\partial U,dx)$ is regular and this implies that the time changed Dirichlet space $(\D,\F^a)$ on $L^2(U\cup\partial U,\Gamma[a])=L^2(U,\Gamma[a])$ is regular too. Its associated Hunt processes is the Brownian motion reflected at $\partial U$ with time change given by the {\it additive functional} associated to $\Gamma[a]$.
\vskip0.1truecm\noindent
ii) It happens that $1\in\F^a$ if and only if the Dirichlet space $(\D,H^{1,2}(U))$ is recurrent on $L^2(U,dx)$ and the multiplier has finite energy $a\in\M(H^{1,2}(U))\cap H^{1,2}(U)$. If $(\D,H^{1,2}(U))$ is recurrent on $L^2(U,dx)$ then $H^{1,2}_e(U)={\rm BL}(U)$ and then $\F^a=\mathcal{G}^a$. This is the case when $U\subset\R^n$ has finite volume for any $n\ge 1$ but also when $U=\R^n$ for $n=1,2$.
\vskip0.1truecm\noindent
iii) If the multiplier has finite energy $a\in\M(H^{1,2}(U))\cap H^{1,2}(U)$ and the Dirichlet space $(\D,H^{1,2}(U))$ is transient on $L^2(U,dx)$, then $1\in\G^a$ while $1\notin\F^a$.
\vskip0.1truecm\noindent
iv) The relationship between the Dirichlet spaces $\F^a$ and $\G^a$ may be discussed within the theory of {\it reflected Dirichlet spaces} (see [CF Ch. 6]).
\end{rem}
\subsection{Fundamental tone of a multiplier}
Here we define the spectral characteristic of multipliers whose properties will be relevant in next sections. For any multiplier $a\in\M(H^{1,2}(U))$ such that ${\rm supp}(\Gamma[a])=U$, we will denote by ${\rm Sp}(\D,\F^a)\subseteq [0,+\infty)$ the spectrum of the Dirichlet space $(\D,\F^a)$ on $L^2(U,\Gamma[a])$, i.e. the spectrum of the self-adjoint, nonnegative operator on $L^2(U,\Gamma[a])$ whose associated quadratic form is $(\D,\F^a)$.
\begin{defn}
The {\it fundamental tone} of a multiplier $a\in\M(H^{1,2}(U))$ with ${\rm supp}(\Gamma[a])=U$, is defined as
\begin{equation}\label{tone}
\mu_1(U,a):=\inf \Bigl\{{\rm Sp}(\D,\F^a)\setminus\{0\}\Bigr\}
\end{equation}
and also denoted by $\mu_1(a)$ if no confusion can arise. Since the spectrum is closed, the fundamental tone belongs to it.
\end{defn}
The following results are aimed to clarify the connections between ergodic properties of the Dirichlet integral $(\D,H^{1,2}(U))$, some spectral property of the Dirichlet form $(\D,\F^a)$ on $L^2(U,\Gamma[a])$ and the multiplier seminorm. As the domain $U$ is, by definition, connected, the Dirichlet space $(\D,H^{1,2}(U))$ on $L^2(U,dx)$ is irreducible so that it is recurrent or transient.
\begin{prop}\label{property1}
Let $a\in\M(H^{1,2}(U))$ be a multiplier with full support ${\rm supp}(\Gamma[a])=U$. Then the following conditions are equivalent:
\vskip0.1truecm\noindent
i) $(\D,H^{1,2}(U))$ is recurrent and the multiplier has finite energy
$a\in\M(H^{1,2}(U))\cap H^{1,2}(U)$;
\vskip0.1truecm\noindent
ii) the zero value is a non degenerate eigenvalue in ${\rm Sp}(\D,\F^a)$.
\end{prop}
\begin{proof}
i) If the Dirichlet space $(\D,H^{1,2}(U))$ on $L^2(U,dx)$ is recurrent then $1\in H^{1,2}_e(U)$ and $\D[1]=0$. Since the multiplier has finite energy, its energy measure $\Gamma[a]$ is finite so that $1\in\F^a$. Hence the constant function $1$ is eigenfunction corresponding to the zero eigenvalue. The non degeneracy follows from the irreducibility of  $(\D,H^{1,2}(U))$ on $L^2(U,dx)$ which in turn is a consequence of the assumption that $U$ is connected (as a domain).\\
ii) On the other hand, let $b\in \F^a= H^{1,2}_e(U)\cap L^2(U,\Gamma[a])$ be an eigenvector corresponding to the zero eigenvalue. Then $\D[b]=0$ so that $b$, by irreducibility, is a constant function belonging to $H^{1,2}_e(U)$ which implies recurrence. Consequently, since $b$ is constant and belongs also to $L^2(U,\Gamma[a])$, it follows that $\Gamma[a]$ is finite so that $a$ has finite energy.
\end{proof}
\begin{prop}\label{property2}
Let $a\in\M(H^{1,2}(U))$ be a multiplier with full support ${\rm supp}(\Gamma[a])=U$.
\vskip0.1truecm\noindent
i) If the fundamental tone is strictly positive $\mu_1(a)>0$, then
\vskip0.1truecm\noindent
i$_1$) $0\in {\rm Sp}(\D,\F^a)$ implies  that Dirichlet space $(\D,H^{1,2}(U))$ is recurrent
\vskip0.1truecm\noindent
i$_2$) $0\notin {\rm Sp}(\D,\F^a)$ implies  that Dirichlet space $(\D,H^{1,2}(U))$ is transient.
\vskip0.1truecm\noindent
ii) If $\mu_1(a)=0$, then $0\in {\rm Sp}_{\rm ess}(\D,\F^a)$ and it is not an eigenvalue.
\end{prop}
\begin{proof}
i$_1$) The condition $\mu_1(a)>0$ and the assumption $0\in {\rm Sp}(\D,\F^a))$ imply that the zero value is isolated in ${\rm Sp}(\D,\F^a)$ so that it is an eigenvalue. If $b\in\F^a$ is an associated eigenfunction then $\D[b]=0$ so that $b$ is constant, by irreducibility. Hence $1\in H^{1,2}_e(U)$ and $D[1]=0$ which implies recurrence.\\
i$_2$) The condition $\mu_1(a)>0$ and the assumption $0\notin {\rm Sp}(\D,\F^a)$ imply $\mu_1(a)=\inf {\rm Sp}(\D,\F^a)$ and the following Poincar\'e inequality
\[
\mu_1(a)\cdot \|b\|^2_{L^2(U,\Gamma[a])}\le\D[b]\qquad b\in \F^a\, .
\]
Let $b\in\H^{1,2}_e(U)$ and $b_n\in H^{1,2}(U)$ a sequence such that $\D[b_n-b_m]\to 0$ as $n,m\to +\infty$, $\D[b_n-b]\to 0$ as $n\to +\infty$ and $b_n\to b$, $dx$-a.e. on $U$. By the Poincar\'e inequality we have $\|b_n-b_m\|_{L^2(U,\Gamma[a])}\to 0$ as $n\to +\infty$ and there exist $b':=\lim_{n\to+\infty}b_n\in L^2(U,\Gamma[a])$. Possibly passing to a suitable subsequence, we have $b_n\to b'$, $\Gamma[a]$-a.e. on $U$ as $n\to+\infty$. Since the support of $\Gamma[a]$ is $U$, we have the identification $b'=b$, $\Gamma[a]$-a.e. on $U$. Passing to the limit $n\to +\infty$ in the Poincar\'e inequalities $\mu_1(a)\cdot \|b_n\|^2_{L^2(U,\Gamma[a])}\le\D[b_n]$, we get that the Poincar\'e inequality is true even for any $b\in\H^{1,2}_e(U)$. Consequently, if for such a function $b$ one has $D[b]=0$, it follows that $b=0$, $\Gamma[a]$-a.e. on $U$ or equivalently $dx$-a.e. on $U$ and this implies transience.\\
ii) By definition of the fundamental tone, if $\mu_1(a)=0$, this value cannot be isolated in ${\rm Sp}(\D,\F^a)$ and then it is not a (discrete or infinitely degenerate) eigenvalue.
\end{proof}

\noindent
When $0\in {\rm Sp}(\D,\F^a)$ is a (non degenerate) eigenvalue, by Proposition \ref{property1}, the multipliers have finite energy. If moreover the fundamental tone of a multiplier is strictly positive $\mu_1(a)>0$, a Poincar\'e-Wirtinger (or spectral gap) inequality holds true
\begin{equation}\label{sp4}
 \mu_1(a)\cdot\|b-p_a(b)\|^2_{L^2(U,\Gamma[a])}\le\D[b]\qquad b\in \F^a
\end{equation}
with $p_a(b)=\int_U b\cdot |\nabla a|^2 dx/ \D [a]$.
When $U$ is a finite volume domain, $H^{1,2}(U)$ contains the constant functions and zero is then a non degenerate eigenvalue of the Dirichlet integral $(\D,H^{1,2}(U))$ on $L^2(U,dx)$. In case {\it zero is an isolated eigenvalue}, a Poincar\'e-Wirtinger (or spectral gap) inequality holds true
\begin{equation}\label{sp3}
 \mu_1(U)\cdot\|b-{\bar b}\|^2_{L^2(U,dx)}\le\D[b]\qquad b\in H^{1,2}(U)
\end{equation}
with ${\bar b}=\int_U b\cdot dx/ |U|$ and the best constant $\mu_1(U)>0$ is called the {\it fundamental tone} of $U$ (the value $\mu_1(U)$ is the first non-zero eigenvalue of the Laplacian. It coincides with the fundamental tone $\mu_1(a)$ of any multiplier $a$ solving the ikonal equation $|\nabla a|=1$ on $U$.
\vskip0.2truecm\noindent
Next result shows that for finite volume domains $U$, the non vanishing of the fundamental tone $\mu_1(U)$ implies the non vanishing of the fundamental tone $\mu_1(a)$ of any multiplier. Recall that if the volume of $U$ is finite, all multipliers have finite energy $\M(H^{1,2}(U))\subseteq H^{1,2}(U)$.
\begin{prop}(Persistence of the spectral gap)\\
Let $U\subset\mathbb{R}^n$ be a finite volume domain with a strictly positive fundamental tone $\mu_1(U)>0$. \\
Then, for any multiplier $a\in\M(H^{1,2}(U))$ such that ${\rm supp}(\Gamma[a])=U$, we have $0\in {\rm Sp}(\D,\F^a)$ and this is a non degenerate eigenvalue. Moreover, the fundamental tone is strictly positive $\mu_1(a)>0$ and, particular,
\[
\frac{1}{{\eta}(a)^2}\cdot\frac{\mu_1(U)}{1+\mu_1(U)}\le\mu_1(a)\qquad a\in\M(H^{1,2}(U))\, .
\]
\end{prop}
\begin{proof}
Since $p_a$ is the orthogonal projection operator onto the subspace of constant functions in  $L^2(U,\Gamma[a])$, for $b\in H^{1,2}(U)$ we have
\[
\|b-p_a(b)\|^2_{L^2(U,\Gamma[a])} \le\|b-{\bar b}\|_{L^2(U,\Gamma[a])}^2\le {\eta}(a)^2\cdot \|b-{\bar b}\|_{H^{1,2}(U)}^2\le  {\eta}(a)^2\cdot(1+\mu_1(U)^{-1})\cdot\D[b]\, .
\]
Since, by Theorem \ref{DF} vi), $H^{1,2}(U)$ is a form core of $(\D,\F^a)$ on $L^2(U,\Gamma[a])$, we get
\begin{equation*} \begin{split}
\mu_1(a)&= \inf_{b\in \F^a\,,\,\D[b]\not=0}\frac{\D[b]}{||b-p_a(b)||_{L^2(U,\Gamma[a])}}
= \inf_{b\in H^{1,2}(U)\,,\,\D[b]\not=0}\frac{\D[b]}{||b-p_a(b)||_{L^2(U,\Gamma[a])}} \\
&\geq \eta(a)^{-2}\big(1+\mu_1(U)^{-1}\big)^{-1}=\frac{1}{\eta(a)^2}\,\frac{\mu_1(U)}{1+\mu_1(U)}\,.
\end{split}\end{equation*}
\end{proof}
\noindent
We end this section with results relating the fundamental tone to the seminorm of a multiplier.
\begin{prop}
Let $a\in\M(H^{1,2}(U))$ be a multiplier with unbounded energy $\D[a]=+\infty$. Then $\mu_1(a)\le\frac{1}{\eta(a)^2}$.
\end{prop}
\begin{proof}
Under the hypothesis, the subspace of constant functions in $L^2(U,\Gamma[a])$ reduces to the zero function only and the the projection $p_a$ vanishes identically so that
\[
\mu_1(a)\cdot\|b\|^2_{L^2(U,\Gamma[a])} \le \D[b]\le \|b\|^2_{H^{1,2}(U)}\qquad b\in H^{1,2}(U)\, .
\]
Since, by Theorem \ref{DF} vi), $H^{1,2}(U)$ is dense in $\F^a$, by the definition of $\eta$ we get the thesis.
\end{proof}
To treat the finite energy case, we need the following result of independent interest. It states that the energy measure of a finite energy multiplier is a finite energy measure.
\begin{prop}
If $a\in\M(H^{1,2}(U))\cap H^{1,2}(U)$ is a finite energy multiplier, then its energy measure $\Gamma[a]$ is a finite energy measure with respect to the Dirichlet space $(\D,H^{1,2}(U))$. Its 1-potential $G(a)\in H^{1,2}(U)$, defined by the identity
\[
\int_U b\,d\Gamma[a]=(b|G(a))_{H^{1,2}(U)}=\D(b|G(a))+(b|G(a))_{L^2(U,dx)}\qquad b\in H^{1,2}(U),
\]
has a norm bounded by
\[
\|G(a)\|_{H^{1,2}(U)}   \le\eta(a)\cdot \sqrt{\D[a]}\, .
\]
In particular, if $U$ has finite volume, then for all multipliers $a\in \M(H^{1,2}(U))$ we have
\[
\|G(a)\|_{H^{1,2}(U)}   \le |U|\cdot\eta(a)^2\, .
\]
\end{prop}
\begin{proof}
The first bound follows from the definitions of $\eta(a)$ and $G(a)$, the estimates
\[
\Bigl|\int_U b\,d\Gamma[a]\Bigr|=|(\nabla a|(\nabla a)b)|\le \sqrt{\D[a]}\cdot\|(\nabla a)b\|\le \sqrt{\D[a]}\cdot\eta(a)\cdot\|b\|_{H^{1.2}(U)}\qquad b\in H^{1,2}(U)
\]
and the density of $H^{1,2}(U)$ in $\F^a$ (Theorem \ref{DF} vi)). The second bound follows from the fact that if $U$ has finite volume, then multipliers have finite energy and $\sqrt{\D[a]}\le\sqrt{|U|}\cdot\eta(a)$.
\end{proof}

\begin{prop}
If $a\in\F\M(H^{1,2}(U))$ is a finite energy multiplier and $G(a)$ is the potential of its energy measure $\Gamma[a]$, then
\[
0\le\eta(a)-\frac{\|G(a)\|_{H^{1,2}(U)}}{\sqrt{\D[a]}}\le\mu_1(a)^{-1/2}\, .
\]
\end{prop}
\begin{proof}
By definition of the fundamental tone $\mu_1(a)\cdot \|b-p_a(b)\|^2_{L^2(U,\Gamma[a])}\le \D[b]$ for $b\in H^{1,2}(U)$ we have
\[
\begin{split}
\|b\|_{L^2(U,\Gamma[a])}&\le \|b-p_a(b)\|_{L^2(U,\Gamma[a])} + \|p_a(b)\|_{L^2(U,\Gamma[a])} \\
&\le \mu_1(a)^{-1/2} \cdot \sqrt{\D[b]}+\|p_a(b)\|_{L^2(U,\Gamma[a])} \\
&\le \mu_1(a)^{-1/2}\cdot \|b\|_{H^{1,2}(U)}+\sqrt{\D[a]}\cdot |p_a(b)| \\
&\le \mu_1(a)^{-1/2}\cdot \|b\|_{H^{1,2}(U)}+\frac{\|G(a)\|_{H^{1,2}(U)}}{\sqrt{\D[a]}}\cdot \|b\|_{H^{1,2}(U)} \\
&\le \Bigl(\mu_1(a)^{-1/2}+\frac{\|G(a))\|_{H^{1,2}(U)}}{\sqrt{\D[a]}}\Bigr)\cdot \|b\|_{H^{1,2}(U)}
\end{split}
\]
and the thesis follows by Theorem \ref{DF} vi).
\end{proof}

\section{Bounded distortion, quasiconformal and conformal maps}

We briefly recall in this section, the definitions of maps with bounded distortion, quasiconformal and conformal maps as well as some result about the integrability of the Jacobian determinant, needed in the forthcoming sections. For the background material we refer to the monographes \cite{IM}, \cite{Re1} and \cite{V}.
\vskip0.2truecm
For a transformation $\gamma\in H^{1,n}_{\rm loc}(U,\mathbb{R}^n)$, its {\it Jacobian determinant} $J_\gamma(x):={\rm det\,}(\gamma^\prime(x))$, defined for $dx$-a.e. $x\in U$, belongs to $L^1_{\rm loc}(U)$. We will need however a subtler result asserting that the local integrability of $J_\gamma $ holds true under milder regularity assumptions on $\gamma$. It is a consequence of the following
\begin{thm}{([IM Theorem 6.3.2])}\label{J1}
Let $B\subset\mathbb{R}^n$ be an open ball and $\gamma:B\to\mathbb{R}^n$ a map lying in the Sobolev class $H^{1,1}(B,\mathbb{R}^n)$. Then there exist a measure zero set $E\subset B$ such that
\begin{equation}
\int_U |J_\gamma(x)|\, dx = \int_{\mathbb{R}^n} N_\gamma(y,U\setminus E)\, dy
\end{equation}
for any measurable $U\subset B$, where the measurable function $N_\gamma (\cdot,F)$ is defined, for a measurable set $F\subset U$, as follows
\[
N_\gamma (y,F)=\sharp\{x\in F:\gamma(x)=y)\}\qquad dy{\rm -a.e.}\,\,y\in \mathbb{R}^n\, .
\]
\end{thm}

\begin{cor}([IM Corollary 6.3.1])\label{J2} Let $U\subset\mathbb{R}^n$ be an open set and $\gamma:U\to\mathbb{R}^n$ a map belonging to $H^{1,1}_{\rm loc}(U,\mathbb{R}^n)\cap L^\infty_{\rm loc}(U,\mathbb{R}^n)$ such that, for some integer $N\ge 1$, there exist a measure zero set $E\subset U$ for which $\gamma:U\setminus E\to\mathbb{R}^n$ is at most $N$-to-$1$.
\par\noindent
Then $J_\gamma\in L^1_{\rm loc}(U)$. In particular, this is the case for local homeomorphisms in $H^{1,1}_{\rm loc}(U,\mathbb{R}^n)$.
\end{cor}
\noindent
We now recall the classes of transformations we are interested in and refer to [IM Chapter 6], [R Chapter I.4] and [V] for details.
\par\noindent
\begin{defn}\label{BD}(Bounded distortion maps, quasiconformal and conformal maps)\par\noindent
A map $\gamma:U\to\mathbb{R}^n$, defined on an open set $U\subseteq\mathbb{R}^n$, is said to have {\it bounded distortion} if satisfies the following requirements:\par\noindent
\item i) it belongs to the Sobolev space $\gamma\in H^{1,1}_{\rm loc}(U,\mathbb{R}^n)$,\par\noindent
\item ii) the Jacobian determinant is locally integrable $J_\gamma\in L^1_{\rm loc}(U)$ and it has constant sign in $U$,\par\noindent
\item iii) there exists $K>0$ such that
\begin{equation}\label{qc}
||\gamma^\prime (x)||\le K\cdot |J_\gamma (x)|^{1/n}\qquad dx-a.e.\,\, x\in U\, .
\end{equation}
The smallest constant $K$ satisfying iii) is called the {\it distortion coefficient} and denoted by $K(\gamma)$. By Hadamard's inequality $|{\rm det\,}A|^{1/n}\le \|A\|$ for matrices  $A\in\mathbb{M}_n(\mathbb{R})$, one has $K(\gamma)\ge 1$.
\vskip0.2truecm\noindent
A map with bounded distortion is said to be {\it quasiconformal} if it is an homeomorphism. A quasiconformal map such that
$K(\gamma)=1$ is called {\it conformal}.
\end{defn}
\begin{rem}
It follows from ii) and iii) that a bounded distortion map necessarily belongs to $ H^{1,n}(U)$. Moreover, it has been proved in [GV] (see also \cite[Chapter 7]{IM} ) that maps with bounded distortion are continuous $\gamma\in C(U,\R^n)$. It follows from Corollary \ref{J2} that for local homeomorphisms $\gamma$ in $H^{1,1}_{\rm loc}(U)$, the local integrability of the Jacobian determinant $J_\gamma$ in Definition \ref{BD}, is automatically satisfied.
\end{rem}
\begin{rem}\label{Alternative}
Alternative definitions of the class of bounded distortion maps may be given, see e.g. \cite{IM}, \cite{Re1} and \cite{V}. We found the adopted choice above convenient for the proof of the main results of Section 5.
\end{rem}

\subsection{M\"obius transformations of $\R^n$ and the conformal group of $\mathbb{S}^n$, $n\ge 3$}
In the next section we will deal with the specific case $U=\R^n$ in dimension $n\ge 3$ and we will consider the M\"obius group $G(\mathbb{R}^n)$ of all M\"obius transformations of $\R^n$. These are the $dx$-a.e. defined, measurable transformations on $\R^n$ which are compositions of a finite number of the elementary ones given by  {\it translations} and {\it rotations},
\[
\gamma_y (x):=x+y\, ,\qquad \gamma_R(x):=Rx\, ,\qquad y\in \mathbb{R}^n\, ,\quad R\in O(n)\, ,
\]
(which are isometries and together generate the Euclidean subgroup), {\it dilations}
\[
\gamma_s(x):=sx\qquad s\in\mathbb{R}_+
\]
and the {\it inversion} $\gamma_i :\mathbb{R}^n\setminus\{0\}\to\mathbb{R}^n\setminus\{0\}$ with respect to the unit sphere centered at $0\in\R^n$
\[
\gamma_i (x):=\frac{x}{|x|^2}\qquad x\in\mathbb{R}^n\setminus\{0\}\, .
\]
Denote by $J_\gamma$ the Jacobian of a transformation $\gamma\in G(\mathbb{R}^n)$. As translations and rotations are isometries they gives $J_{\gamma_y} (x)=1$, $|J_{\gamma_R}(x)|=1$ for all $x\in\mathbb{R}^n$. For dilations and the inversion it is easy to check that
\[
J_{\gamma_s} (x)=s^n\, ,\qquad J_{\gamma_i} (x)=-|x|^{-2n}\, ,\qquad x\in\mathbb{R}^n\setminus\{0\}\, .
\]
Transformations $\gamma\in G(\R^n)$ are invertible and one has $J_{\gamma^{-1}}(y)^{-1}=J_\gamma (\gamma^{-1}(y))$ $dy$-a.e. on $\mathbb{R}^n$. Notice that all of them are conformal maps.
\vskip0.2truecm\noindent
It is well known and easy to check that $G(\R^n)$ acts isometrically on the Lebesgue spaces:
\begin{lem}\label{grouprepresentation}
For any fixed $p\in [1,+\infty]$ and $\gamma\in G(\R^n)$, the following is a well defined, isometric transformation
\[
\gamma_p^* :L^p (\mathbb{R}^n,dx)\to L^p (\mathbb{R}^n,dx)\qquad \gamma_p^* (f)(y):=|J_{\gamma^{-1}}(y)|^{1/p}\cdot  f(\gamma^{-1}(y))\qquad dy-a.e. \quad y\in \R^n\, .
\]
Isometric representations of the conformal group in Lebesgue spaces are given by
\[
\pi_p:G(\R^n)\to \mathcal{B}(L^p(\R^n,dx))\qquad \pi_p(\gamma):=\gamma_p^*\, .
\]
\end{lem}
\noindent
In case $p=\infty$ we have $\gamma_\infty^*(f)=f\circ\gamma^{-1}$ for $f\in L^\infty(\R^n,dx)$ and $\gamma\in G(\R^n)$.
\vskip0.2truecm\noindent
M\"obius transformations of $\R^n$ are $dx$-a.e. defined, measurable maps on the space $\R^n$, endowed with the Lebesgue measure.  They can be understood, in a natural way, as homeomorphisms of the one-point compactification ${\bar\R^n}:=\R^n\cup\{\infty\}$ so that any $\gamma\in G(\R^n)$ reduces to a conformal diffeomorphism on $\bar\R^n\setminus \{\infty,\gamma^{-1}(\infty)\}$.\\
M\"obius transformations can also be seen as homeomorphisms of the unit sphere $\mathbb{S}^n\subset \R^{n+1}$: the inverse stereographic projection $S:\R^n\to \mathbb{S}^n$ provides an isomorphism $\gamma\mapsto \widetilde\gamma :=S\circ\gamma\circ S^{-1}$ between the M\"obius group $G(\R^n)$ and the group $G(\mathbb{S}^n)$ of conformal diffeomorphisms of $\mathbb{S}^n$  (see [LL Chapter 4.4] for details).\\
The celebrated rigidity theorem of J. Liouville \cite{L1}, \cite{L2} states that, in striking contrast with plane mappings, in dimension $n\ge 3$, the only conformal mappings on domains $D\subseteq\R^n$ are restrictions of M\"obius transformations to $D$. Liouville's proof required the mapping to be at least $C^3$ and since then several generalizations of the result have been provided requiring lesser regularity. In particular we will make use of the Gehring's version \cite{Ge} for 1-quasiconformal mappings and the Reshetnyak's one \cite{Re2} for 1-quasiregular mappings.
\section{M\"obius transformations as automorphims of a multipliers algebra}
All along this section, we suppose $n\geq 3$ and we show that the action $\pi_\infty$ of the M\"obius group $G(\R^n)$ on the algebra $L^\infty(\R^n,dx)$ restricts to an isometric action of $G(\R^n)$ on the multipliers algebra $\M(H^{1,2}_e(\R^n))$ and that the spectrum of the Dirichlet integral $\D$ on spaces $L^2(\R^n,\Gamma[a])$ is the same for all multipliers $a\in\M(H^{1,2}_e(\R^n))$ belonging to the same $\pi_\infty$-orbit of $G(\R^n)$.
\subsection{Green operator conformal covariance}
Denoting by $|\mathbb{S}^{n-1}|$ the measure of the unit sphere $\mathbb{S}^{n-1}\subset\mathbb{R}^n$, recall that the Green function on $\mathbb{R}^n$, for $n\ge 3$, is defined by
\begin{equation}
G(x,y):=c_n|x-y|^{2-n}\qquad x,y\in \mathbb{R}^n\, ,\quad x\neq y
\end{equation}
where $c_n :=[(n-2)|\mathbb{S}^{n-1}|]^{-1}$, in such a way that the restriction $G_y (x):=G(x,y)$ is the fundamental solution of the Poisson equation $-\Delta G_y = \delta_y$, $y\in \mathbb{R}^n$, associated to the Laplacian $\Delta:=\sum_{k=1}^n \partial^2/\partial_k^2$. The heat semigroup $e^{t\Delta}$ is transient for $n\ge 3$ and its resolvent Green operator $G:=\Delta^{-1}$ is an integral operator given by
\[
G(f)(x)=(-\Delta^{-1}f)(x)=\int_{\mathbb{R}^n} G(x,y)f(y)dy = (G_0\ast f)(x)\qquad dx-a.e.\quad x\in\R^n\,
\]
with $G_0(y)=c_n|y|^{2-n}$, $y\in \R^n\setminus\{0\}$.
A first contact between conformal geometry and Laplace operator is revealed by the following
\begin{prop}(Green operator conformal covariance)\label{green}\par\noindent
Setting for $n\ge 3$, $p:=\frac{2n}{n+2}$ and $r:=\frac{2n}{n-2}$ (Sobolev exponent), we have that the representations $\pi_p$ and $\pi_r$ of the conformal group are intertwined by the actions of the Green operator on the spaces $L^p(\R^n,dx)$ and $L^r(\R^n,dx)$
\begin{equation}
G(\gamma_p^* (f))=\gamma^*_r (G(f))\qquad f\in L^p (\mathbb{R}^n,dx)\, ,\quad \gamma\in G(\R^n)\, .
\end{equation}
Here we have implicitly admitted the fact that, by the Hardy-Littlewood-Sobolev inequality (see below), the Green operator $G:L^p (\mathbb{R}^n,dx)\to L^r(\mathbb{R}^n,dx)$ is well defined (and bounded).
\end{prop}
\begin{proof}
It is enough to check the identity for the generators of the conformal group. As translations and rotations are isometries, their Jacobian is the unit constant function and their actions preserve the Lebesgue measure. The identity is then consequence of the fact that the Green kernel is a function of the Euclidean distance of $\mathbb{R}^n$.
\par\noindent
In case of scalings, $\gamma =\gamma_s$ for some $s\in\mathbb{R}_+$, we have
\begin{equation}
\begin{split}
G(\gamma_p^* (f))(x)&=c_n\int_{\mathbb{R}^n} |x-y|^{(2-n)}f(\gamma^{-1}(y))J_{\gamma^{-1}}(y)^{1/p}\, dy \\
&=s^{-n/p}c_n\int_{\mathbb{R}^n} |x-y|^{(2-n)}f(\gamma^{-1}(y))\, dy \\
&=s^{n-n/p}c_n\int_{\mathbb{R}^n} |x-\gamma(y')|^{(2-n)}f(y')\, dy' \\
&=s^{n-n/p}c_n\int_{\mathbb{R}^n} |x-sy'|^{(2-n)}f(y')\, dy' \\
&=s^{n-n/p}c_n\int_{\mathbb{R}^n} s^{(2-n)}|s^{-1}x-y'|^{(2-n)}f(y')\, dy' \\
&=s^{2-n/p}c_n\int_{\mathbb{R}^n} |\gamma^{-1}(x)-y'|^{(2-n)}f(y')\, dy' \\
&=s^{-n/r}G(f)(\gamma^{-1}(x)) \\
&=J_{\gamma^{-1}}(x)^{1/r}G(f)(\gamma^{-1}(x)) \\
&=\gamma^*_r (G(f))(x)\, .
\end{split}
\end{equation}
In the case of the inversion $\gamma=\gamma_i$ we use the property $|\,\gamma (x)-\gamma(y)|=\frac{|x-y|}{|x|\cdot|y|}$ to compute
\begin{equation}
\begin{split}
G(\gamma_p^* (f))(x)&=c_n\int_{\mathbb{R}^n} |x-y|^{(2-n)}f(\gamma^{-1}(y))|J_{\gamma^{-1}}(y)|^{1/p}\, dy \\
&=c_n\int_{\mathbb{R}^n} |x-y|^{(2-n)}f(\gamma^{-1}(y))|J_{\gamma}(\gamma^{-1}(y))|^{-1/p}\, dy \\
&=c_n\int_{\mathbb{R}^n} |x-\gamma (y)|^{(2-n)}f(\gamma^{-1}(y))|J_{\gamma}(y')|^{-1/p}|J_{\gamma}(y')|\, dy' \\
&=c_n\int_{\mathbb{R}^n} |\gamma(\gamma^{-1}(x))-\gamma (y)|^{(2-n)}f(y')|J_{\gamma}(y')|^{1-1/p}\, dy' \\
&=c_n\int_{\mathbb{R}^n} |\gamma(\gamma_i^{-1}(x))-\gamma (y)|^{(2-n)}f(y')|J_{\gamma}(y')|^{1-1/p}\, dy' \\
&=c_n\int_{\mathbb{R}^n} |\gamma^{-1}(x)|^{(n-2)}|y'|^{(n-2)}|\gamma^{-1}(x)-y'|^{(2-n)}f(y')|J_{\gamma}(y')|^{1-1/p}\, dy' \\
&=|\gamma^{-1}(x)|^{(n-2)}c_n\int_{\mathbb{R}^n} |y'|^{(n-2)}|J_{\gamma}(y')|^{1-1/p}|\gamma^{-1}(x)-y'|^{(2-n)}f(y')\, dy' \\
&=|\gamma^{-1}(x)|^{(n-2)}c_n\int_{\mathbb{R}^n} |y'|^{(n-2)}(|y'|^{-2n})^{1-1/p}|\gamma^{-1}(x)-y'|^{(2-n)}f(y')\, dy' \\
&=|\gamma^{-1}(x)|^{(n-2)}c_n\int_{\mathbb{R}^n} |y'|^{(n-2-2n(1-1/p))}|\gamma^{-1}(x)-y'|^{(2-n)}f(y')\, dy' \\
&=|x|^{(2-n)}c_n\int_{\mathbb{R}^n} |\gamma^{-1}(x)-y'|^{(2-n)}f(y')\, dy' \\
&=(|x|^{-2n})^{\frac{n-2}{2n}}G(f)(\gamma^{-1}(x)) \\
&=|J_{\gamma^{-1}}(x)|^{1/r}G(f)(\gamma^{-1}(x)) \\
&=\gamma^*_{r}(G(f))(x)\, .
\end{split}
\end{equation}
\end{proof}

\subsection{Conformal invariance of the Hardy-Littlewood-Sobolev inequality}
Consider the functions $g_\lambda (x):=|x|^{-\lambda}$, defined for $x\in\R^n\setminus\{0\}$ and $\lambda\in (0,n)$, and the Riesz potential operators defined by
\[
G_\lambda (f)(x):=(g_\lambda\ast f)(x)=\int_{\mathbb{R}^n}f(y)|x-y|^{-\lambda}\, dy\qquad dx-a.e.\,\, x\in\R^n\, .
\]
For $\lambda =(n-2)$ the Riesz potential is proportional to the Green operator $G=c_n\cdot G_{n-2}$.
\vskip0.1truecm\noindent
By the Hardy-Littlewood-Sobolev inequality  (see \cite{LL} Chapter 4.3)
\begin{equation}
\int_{\mathbb{R}^n} \int_{\mathbb{R}^n}  f(x)|x-y|^{-\lambda}h(y)dxdy\le {\rm const.}\cdot\|f\|_p\cdot\|h\|_q\qquad f\in L^p (\mathbb{R}^n,dx), h\in L^q (\mathbb{R}^n,dx)
\end{equation}
corresponding to parameters $\lambda\in (0,n)$ and $p,q\in(1,+\infty]$ such that $\frac{1}{p}+\frac{\lambda}{n}+\frac{1}{q}=2$,
it follows that the Riesz potentials are bounded operator from $L^q$ to $L^{p'}$ and from $L^p$ to $L^{q'}$, where $p'$, resp. $q'$ denotes the exponent conjugate to $p$, resp. $q$.
\vskip0.2truecm\noindent
In particular, choosing $\lambda =(n-2)$ and $q=\frac{2n}{n+2}$, we have $p=\frac{2n}{n+2}$, $p'=\frac{2n}{n-2}$ and then $G_{n-2}\in\mathcal{B}(L^{\frac{2n}{n+2}},L^{\frac{2n}{n-2}})$, while, choosing $\lambda =(n-1)$ and $q=\frac{2n}{n+2}$, we have $p=p'=2$ and then $G_{n-1}\in\mathcal{B}(L^{\frac{2n}{n+2}},L^2)$. Notice that $r:=\frac{2n}{n-2}$ is the Sobolev exponent.
\vskip0.2truecm\noindent
For future reference, we collect below a well known invariance property of the HLS inequalities as well as some consequences  that we will have occasion to exploit later on.
\begin{thm}(HLS conformal invariance)\label{HLSinvariance}
\item i) The HSL functional corresponding to the values $\lambda\in (0,n)$, $p=q:=\frac{2n}{2n-\lambda}$ of the parameters
\begin{equation}
I(f,h):=\int_{\mathbb{R}^n} \int_{\mathbb{R}^n}  f(x)|x-y|^{-\lambda}g(y)dxdy\qquad f,g\in L^p (\mathbb{R}^n,dx)
\end{equation}
is invariant under the representation $\pi_p$ of the M\"obius group $G(\mathbb{R}^n)$
\begin{equation}
I(\gamma^*_p (f),\gamma^*_p (g))=I(f,g)\qquad f,g\in L^p (\mathbb{R}^n,dx)\, ;
\end{equation}
\item ii) a natural, dense, continuous embedding $H^{1,2}_e(\R^n)\subset L^{\frac{2n}{n-2}}(\R^n,dx)$ holds true;
\item iii) the Riesz potential operator $G_{n-2}$ is a well defined bounded map with dense range
\[
G_{n-2}:L^{\frac{2n}{n+2}}(\R^n,dx)\to H^{1,2}_e(\R^n)\, .
\]
\end{thm}
\begin{proof}
i) This is a well known property whose proof may found in \cite{LL} Chapter 4.5. To prove the subsequent items set $p:=\frac{2n}{n+2}$, $r:=\frac{2n}{n-2}$.\\ ii) The Sobolev inequality
\[
\|b\|^2_r\le S_n\cdot \D[b]\qquad b\in H^{1,2}(\R^n)\, ,
\]
provides a natural embedding $H^{1,2}(\R^n)\subseteq L^r(\R^n,dx)$. Since functions in $H^{1,2}_e(\R^n)$ are $dx$-a.e. limits of $\sqrt{\D}$-Cauchy sequences in $H^{1,2}(\R^n)$, the Sobolev inequality above holds true for any $b\in H^{1,2}_e(\R^n)$ with the same Sobolev constant $S_n $ and one obtains an embedding $H^{1,2}_e(\R^n)\subset L^r(\R^n,dx)$. The continuity of this embedding follows from the extension of the Sobolev inequality to $H^{1,2}_e(\R^n)$ and the fact that $\sqrt{\D}$ is just the norm of $H^{1,2}_e(\R^n)$. Since the subspace $C^\infty_c(\R^n)\subset H^{1,2}_e(\R^n)$ is dense in $L^r(\R^n,dx)$ the embedding has dense range.\\
iii) We already noticed that for $f\in L^p(\R^n,dx)$ we have $G_{n-2}(f)\in L^r(\R^n,dx)\subset L^2_{\rm loc}(\R^n,dx)$. Taking into account the pointwise bound
\[
|\nabla G_\lambda (f)|=|\nabla g_\lambda\ast  f|\le |\nabla g_\lambda|\ast |f|=\lambda g_{\lambda+1}\ast|f|=\lambda G_{\lambda +1} (|f|)
\]
for the parameter $\lambda =(n-2)$ and using the boundedness of the Riesz potential $G_{n-1}$ from $L^p(\R^n,dx)$ to $L^2(\R^n,dx)$ we have
\[
\D[G_{n-2}(f)]=\int_{\mathbb{R}^n}|\nabla G_{n-2}(f)(x)|^2\, dx\le (n-2)^2 \|G_{n-1}(|f|)\|_2^2\le (n-2)^2\cdot \|G_{n-1}\|^2_{p\to 2} \cdot \|f\|^2_p\, ,
\]
so that $G_{n-2}(f)\in {\rm BL}(\R^n)\cap L^r(\R^n,dx)$. Since however,  ${\rm BL}(\R^n)=\{u+c:u\in H^{1,2}_e(\R^n),\,\, c\in\mathbb{C}\}$, we have $G_{n-2}(f)=u+c$ for some $u\in H^{1,2}_e(\R^n)$ and $c\in\mathbb{C}$. This implies $c=G_{n-2}(f)-u\in L^r(\R^n,dx)$ so that $c=0$ and $G_{n-2}(f)=u\in H^{1,2}_e(\R^n)$. The bound above implies that $\|G_{n-2}\|_{L^p\to H^{1,2}_e}\le (n-2)\cdot \|G_{n-1}\|_{p\to 2}<\infty$.
\end{proof}

\subsection{Conformal invariance of the Dirichlet integral}
The main result of this section is based on the following known invariance of the energy functional with respect to the action of the M\"obius group. For reader's convenience we provide a detailed proof.
\begin{thm}\label{energyinvariance}
The Dirichlet integral of the Euclidean space $\R^n$, $n\ge 3$,
\[
\D[f]=\int_{\mathbb{R}^n} |\nabla f(x)|^2\, dx\qquad f\in H^{1,2}_e(\mathbb{R}^n)
\]
is invariant under the action $\pi_r$ of the conformal group for the Sobolev exponent $r=\frac{2n}{n-2}$
\begin{equation}
\D[\gamma_r^* (f)]=\D[f]\qquad f\in H^{1,2}_e(\mathbb{R}^n)\, ,\quad \gamma\in G(\mathbb{R}^n)\, .
\end{equation}
In particular, one has the embedding $H^{1,2}_e(\R^n)\subset L^r(\R^n,dx)$ and for $\gamma\in G(\R^n)$, the isometry $\gamma_r^*:L^r(\R^n,dx)\to L^r(\R^n, dx)$ restricts to a unitary map $\gamma_r^*:H^{1,2}_e(\R^n)\to H^{1,2}_e(\R^n)$.
\end{thm}
\begin{proof}
By Theorem \ref{HLSinvariance} ii), we may re-write the HLS functional, corresponding to the parameters $\lambda=(n-2)$, $p=q=\frac{2n}{n+2}$, by the Green operator $G=c_n\cdot G_{n-2}$ as follows
\begin{equation}
\begin{split}
I(f,f):&=\int_{\mathbb{R}^n} \int_{\mathbb{R}^n}  f(x)|x-y|^{2-n}f(y)dxdy \\
&= \frac{1}{c_n}\int_{\mathbb{R}^n}  f(x)(G(f))(y)dx \\
&=-\frac{1}{c_n}\int_{\mathbb{R}^n}  \Delta(G(f))(x)(G(f))(y)dx \\
&=\frac{1}{c_n}\int_{\mathbb{R}^n}  (\nabla (G(f))(x)\cdot (\nabla (G(f))(x)\, dx \\
&=\frac{1}{c_n}\D[G(f)] \qquad f\in L^p (\mathbb{R}^n)\, . \\
\end{split}
\end{equation}
Since $\gamma_p^* (f)\in L^p(\R^n,dx)$, by Theorem \ref{HLSinvariance} i) we have also that $|\nabla G(\gamma_p^* (f))|\in L^2(\R^n,dx)$ and, by the covariance of the Green operator as in Proposition \ref{green}, $G(\gamma_p^* (f))=\gamma_r^* (G(f))$ so that $|\nabla \gamma_r^* (G(f))|\in L^2(\R^n,dx)$. The invariance of the HLS functional, Theorem \ref{HLSinvariance} i). then implies
\begin{equation}
\begin{split}
\D[G(f)] &= c_n I(f,f) \\
&= c_n I(\gamma^*_p (f),\gamma_p^* (f)) \\
&= \D[G(\gamma^*_p (f))] \\
&= \D[\gamma^*_r (G(f))]\qquad f\in L^{\frac{2n}{n+2}} (\R^n,dx)\, . \\
\end{split}
\end{equation}
To conclude we have to show that $\mathcal{C}:=\{G(f):f\in L^{\frac{2n}{n+2}} (\R^n,dx)\}$ is dense in the Dirichlet space $H^{1,2}(\mathbb{R}^n)$.
Set $L=-\Delta$ and fix  $g\in H^{1,2}(\mathbb{R}^n)\cap L^{\frac{2n}{n+2}} (\R^n,dx)$: we have $f_\varepsilon :=L(I+\varepsilon L)^{-1}g =\varepsilon^{-1}[I-(I+\varepsilon L)^{-1}]g\in L^{\frac{2n}{n+2}} (\R^n,dx)$, because $(I+\varepsilon L)^{-1} (L^{\frac{2n}{n+2}} (\R^n,dx)\subseteq L^{\frac{2n}{n+2}}(\R^n,dx)$. Consider $g_\varepsilon :=G(f_\varepsilon)\in\mathcal{C}$: since $g_\varepsilon =(I+\varepsilon L)^{-1}g$ and, by assumption, $g\in H^{1,2}(\mathbb{R}^n)$, we have that $g_\varepsilon\to g$ in $H^1(\mathbb{R}^n)$ so that $\E[g_\varepsilon]\to \E[g]$. By the Sobolev inequality
\[
\|h\|^2_{L^r}\le c\cdot \E[h]\qquad h\in H^{1,2}(\mathbb{R}^n)\, ,
\]
$g_\varepsilon\to g$ in $L^r(\R^n,dx)$ and, by continuity, $\gamma_r^* (g_\varepsilon)\to \gamma_r^* (g)$ in $L^r(\R^n,dx)$ too. Since by (8.16), $\gamma_r^* (g_\varepsilon)\in H^{1,2}(\mathbb{R}^n)$ is a Cauchy sequence converging in $H^{1,2}(\mathbb{R}^n)$, again by the Sobolev inequality
we may identify its limit with $\gamma_r^* (g)$ so that $\E[g]=\E[\gamma_r^* (g)]$ for all $g\in H^{1,2}(\mathbb{R}^n)\cap L^{\frac{2n}{n+2}} (\R^n,dx)$. For $h\in H^{1,2}(\mathbb{R}^n)$ consider $h_t :=e^{-tL}g$. As the heat semigroup is ultracontractive, $h_t\in H^{1,2}(\mathbb{R}^n)\cap L^{\frac{2n}{n+2}} (\R^n,dx)$ for all $t>0$. Reasoning as above we have
\[
\E[h]=\lim_{t\to 0}\E[h_t]=\lim_{t\to 0}\E[\gamma_r^* (h_t)]=\E[\gamma_r^* (h)]\qquad h\in H^{1,2}(\mathbb{R}^n)\, .
\]
\end{proof}

\subsection{Action of the M\"obius group $G(\mathbb{R}^n)$ on the extended multipliers algebra.\,\,}
In this paragraph we prove the main results of this section. The first concerns the isometric action of the M\"obius group on the multipliers algebra of the extended Dirichlet space while the second shows the stability of the spectrum of the Dirichlet integral $\D$ with respect to spaces $L^2(\R^n,\Gamma[a])$ corresponding to multipliers within the same $\pi_\infty$-orbit.
\vskip0.2truecm\noindent
Since the Dirichlet space  $(\D,H^{1,2}(\R^n))$ is transient so that the extended space $H^{1,2}_e(\R^n)$ is a Hilbert space in the energy norm, in addition to the the multipliers algebra $\M(H^{1,2}(\R^n))$ of the Dirichlet space $H^{1,2}(\R^n)$ treated so far, we may consider the multipliers algebra $\M(H^{1,2}_e(\R^n))$ of the extended Dirichlet space $H^{1,2}_e(\R^n)$.
\begin{defn}(Multipliers of the extended Dirichlet space)
A multiplier of the extended Dirichlet space $H^{1,2}_e (\mathbb{R}^n)$, $n\ge 3$, is a measurable function $a$ such that $ab\in H^{1,2}_e (\mathbb{R}^n)$ for all $b\in H^{1,2}_e (\mathbb{R}^n)$. Multipliers form an algebra denoted by $\M(H^{1,2}_e (\mathbb{R}^n))$.
\end{defn}
\begin{prop}
\vskip0.2truecm\noindent
i) Multipliers $a\in \M(H^{1,2}_e (\mathbb{R}^n))$ are bounded operators on the Hilbert space $H^{1,2}_e (\mathbb{R}^n)$ for $n\ge 3$.
\par\noindent
ii) The multipliers algebra of $H^{1,2}_e (\mathbb{R}^n)$ is a subalgebra of the multipliers algebra of $H^{1,2} (\mathbb{R}^n)$
\[
\M(H^{1,2}_e (\mathbb{R}^n))\subseteq \M(H^{1,2} (\mathbb{R}^n))\, .
\]
\end{prop}
\begin{proof}
i) Since $n\ge 3$, the Dirichlet space $(\D, H^{1,2}(\mathbb{R}^n))$ on $L^2(\mathbb{R}^n,dx)$ is transient and the extended Dirichlet space
$H^{1,2}_e (\mathbb{R}^n)$ is a Hilbert space under the norm $\sqrt\D$. Let $b_n,b,b^\prime\in H^{1,2}_e(\mathbb{R}^n)$ such that $\D[b_n-b]\to 0$ and $\D[ab_n-b^\prime]\to 0$ as $n\to +\infty$. By [CF Theorem 2.1.5], it follows that, possibly passing to a suitable subsequence, we have that $b_n\to b$ and $ab_n\to b^\prime$ pointwise $dx$-a.e. in $\mathbb{R}^n$ so that $b^\prime = ab$. Thus, the multiplication operator by the multiplier $a$ is closed and since it is everywhere defined on the Hilbert space $H^{1,2}_e (\mathbb{R}^n)$ is bounded, by the Closed Graph Theorem.
\par\noindent
ii) Consider $a\in \M(H^{1,2}_e (\mathbb{R}^n))$ and $b\in H^{1,2}(\mathbb{R}^n)$. Since $H^{1,2}(\mathbb{R}^n)=H^{1,2}_e (\mathbb{R}^n)\cap L^2(\mathbb{R}^n,dx)$ and $a$ lies in $L^\infty(\mathbb{R}^n,dx)$, we have $ab\in H^{1,2}_e(\mathbb{R}^n)\cap L^2(\mathbb{R}^n,dx)=H^{1,2}(\mathbb{R}^n)$.
\end{proof}

\begin{thm}\label{conformalactiononmultipliers}(Conformal flow of multipliers)\\
The M\"obius group $G(\mathbb{R}^n)$, $n\ge 3$, acts isometrically, by restriction of the representation $\pi_\infty$, on the multiplier algebra $\M(H^{1,2}_e (\mathbb{R}^n))$, i.e. we have
\[
\gamma\in G(\mathbb{R}^n)\, ,\quad a\in\M(H^{1,2}_e(\mathbb{R}^n))\qquad \Rightarrow\qquad \gamma_\infty^*(a)=a\circ\gamma^{-1}\in\M(H^{1,2}_e(\mathbb{R}^n))
\]
and
\[
\|a\circ\gamma^{-1}\|_{\M(H^{1,2}_e(\mathbb{R}^n))}= \|a\|_{\M(H^{1,2}_e(\mathbb{R}^n))}\, .
\]
\end{thm}
\begin{proof}
Recall that, by Sobolev inequality, $H^{1,2}_e(\mathbb{R}^n)\subset L^r(\mathbb{R}^n,dx)$. For $a\in\M(H^{1,2}_e(\mathbb{R}^n))$ and $b\in H^{1,2}_e(\mathbb{R}^n)$ we have $ab\in H^{1,2}_e(\mathbb{R}^n)\subset L^r(\mathbb{R}^n,dx)$ and then
\[
\gamma_r^* (ab)=J_{\gamma^{-1}}^{1/r}\cdot((ab)\circ\gamma^{-1})=(a\circ\gamma^{-1})\cdot J_{\gamma^{-1}}^{1/r}\cdot(b\circ\gamma^{-1})=(a\circ\gamma^{-1})\cdot\gamma_r^* (b)=\gamma_\infty^*(a)\cdot\gamma_r^* (b)\in L^r(\mathbb{R}^n,dx)\, .
\]
Hence, by Theorem \ref{energyinvariance}, for all $b\in H^{1,2}(\mathbb{R}^n)$ we have
\begin{equation}
\begin{split}
\|(a\circ\gamma^{-1})\gamma_r^* (b)\|_{H^{1,2}_e(\mathbb{R}^n)}^2&=\D[(a\circ\gamma^{-1})\gamma_r^* (b)] \\
&=\D[\gamma_r^* (ab)] \\
&=\D[ab] \\
&=\|ab\|_{H^{1,2}_e(\mathbb{R}^n)}^2 \\
&\le\|a\|_{\M(H^{1,2}_e(\mathbb{R}^n))}^2\cdot \|b\|_{H^{1,2}_e(\mathbb{R}^n)}^2 \\
&\le\|a\|_{\M(H^{1,2}_e(\mathbb{R}^n))}^2\cdot \D[b] \\
&=\|a\|_{\M(H^{1,2}_e(\mathbb{R}^n))}^2\cdot \D[\gamma_r^* (b)] \\
&=\|a\|_{\M(H^{1,2}_e(\mathbb{R}^n))}^2\cdot \|\gamma_r^* (b)\|_{H^{1,2}_e(\mathbb{R}^n)}^2 \, .
\end{split}
\end{equation}
Since, by Theorem \ref{energyinvariance}, $\gamma_r^*$ is a surjective isometry on $H^{1,2}_e(\mathbb{R}^n)$, we have that $a\circ\gamma^{-1}$ is a multiplier with $\|a\circ\gamma^{-1}\|_{\M(H^{1,2}_e(\mathbb{R}^n))}\le
\|a\|_{\M(H^{1,2}_e(\mathbb{R}^n))}$. The conclusion of the proof is attained replacing $a$ with $a\circ\gamma$ and applying the above result to
$\gamma^{-1}$
\[
\|a\|_{H^{1,2}_e(\mathbb{R}^n)}\le \|a\circ\gamma\|_{H^{1,2}_e(\mathbb{R}^n)}=\|a\|_{H^{1,2}_e(\mathbb{R}^n)}\, .
\]
\end{proof}
The second main result of this section is based on the following one by which, for any M\"obius transformation $\gamma\in G(\R^n)$, the map $\gamma_r^*$ establishes a unitary equivalence between the Hilbert spaces $L^2(\R^n,\Gamma[a\circ\gamma])$ and $L^2(\R^n,\Gamma[a])$ for any multiplier $a\in\M(H^{1,2}_e(\R^n))$.
\begin{prop}\label{L2unitary}(Flow of energy measures of multipliers)\\
Let $\gamma\in G(\R^n)$ be a M\"obius transformation of $\R^n$, $n\ge 3$, and let $a\in\M(H^{1,2}_e(\R^n))$ be a multiplier of the extended Dirichlet space $H^{1,2}_e(\R^n)$. Then the map
\[
\gamma_r^*:H^{1,2}_e(\R^n)\to H^{1,2}_e(\R^n)\qquad \gamma_r^*(b):=|J_{\gamma^{-1}}|^{1/r}\cdot (b\circ\gamma^{-1})
\]
obtained in Theorem \ref{energyinvariance}, extends to a unitary map from $L^2(\R^n,\Gamma[a\circ\gamma])$ onto $L^2(\R^n,\Gamma[a])$
\[
(\gamma_r^*(b_1)|\gamma_r^*(b_2))_{L^2(\R^n,\Gamma[a])}=(b_1|b_2)_{L^2(\R^n,\Gamma[a\circ\gamma])} \qquad b_1,b_2\in L^2(\R^n,\Gamma[a\circ\gamma])\, .
\]
In other words, the image by $\gamma$ of the energy measure $\Gamma [a\circ\gamma]$ is the measure $|J_{\gamma^{-1}}|^{2/r} \cdot\Gamma[a]$.
\end{prop}
\begin{proof}
For $b\in H^{1,2}_e(\R^n)$, by Theorem \ref{energyinvariance} above, we have $\gamma_r^*(b)\in H^{1,2}_e(\R^n)$. Since, by Theorem \ref{conformalactiononmultipliers} above, $a\circ\gamma$ is a multiplier of $H^{1,2}_e(\R^n)$, we have $b\in L^2(\R^n,\Gamma[a\circ\gamma])$ and $\gamma_r^*(b)\in L^2(\R^n,\Gamma[a])$. Moreover,
\[
\begin{split}
\|b\|^2_{L^2(\R^n,\Gamma[a\circ\gamma])} &=\int_{\R^n} |b(x)|^2\cdot |\nabla(a\circ\gamma)(x)|^2\cdot dx\\
&=\int_{\R^n} |b(x)|^2\cdot |\gamma^{\prime}(x)^t\cdot (\nabla a(\gamma(x)))|^2\cdot dx\\
&=\int_{\R^n} |b(x)|^2\cdot |J_\gamma (x)|^{2/n}\cdot |\nabla a(\gamma(x))|^2\cdot dx\\
&=\int_{\R^n} |b(x)|^2\cdot |J_\gamma (x)|^{2/n-1}\cdot |\nabla a(\gamma(x))|^2\cdot  |J_\gamma (x)|\cdot dx\\
&=\int_{\R^n} |b(\gamma^{-1}(y))|^2\cdot |J_\gamma (\gamma^{-1}(y))|^{2/n-1}\cdot |\nabla a(y)|^2\cdot  dy\\
&=\int_{\R^n} |b(\gamma^{-1}(y))|^2\cdot |J_{\gamma^{-1}} (y)|^{1-2/n}\cdot |\nabla a(y)|^2\cdot  dy\\
&=\int_{\R^n} |b(\gamma^{-1}(y))|^2\cdot |J_{\gamma^{-1}} (y)|^{2/r}\cdot |\nabla a(y)|^2\cdot  dy\\
&=\int_{\R^n} | |J_{\gamma^{-1}} (y)|^{1/r}\cdot b(\gamma^{-1}(y))|^2\cdot |\nabla a(y)|^2\cdot  dy\\
&=\int_{\R^n} |\gamma_r^*(b)(y)|^2\cdot |\nabla a(y)|^2\cdot  dy\\
&=\|\gamma_r^*(b)\|^2_{L^2({\R^n},\Gamma[a])}\, .
\end{split}
\]
By polarization, we deduce the stated identity for any $b_1,b_2\in L^2({\R^n},\Gamma[a\circ\gamma])$.
\end{proof}
\vskip0.2truecm\noindent
The following are the main results of this section showing that the Dirichlet integral provides unitarily isomorphic, closed quadratic forms with respect to the energy measures of multipliers, belonging to the same orbit of the M\"obius group. For the sake of clarity we first prove a "global" version that will be later  localized on domains with absorbing (or Dirichlet) boundary conditions. Finally we organize the statement of the general version.
\begin{thm}\label{global}
Let $a\in\M(H^{1,2}_e(\R^n))$ be a multiplier with ${\rm supp}(\Gamma[a])=\R^n$. Then setting
\[
\F^a:=H^{1,2}_e(\R^n)\cap L^2(\R^n,\Gamma[a])
\]
we have
\vskip0.1truecm\noindent
\item i) $(\D,\F^a)$ is a regular Dirichlet form on $L^2(\R^n, \Gamma[a])$;
\vskip0.1truecm\noindent
\item ii) for any fixed M\"obius transformation $\gamma\in G(\R^n)$, the Dirichlet forms
\[
(\D,\F^a)\quad{\rm on}\quad L^2(\R^n, \Gamma[a])\qquad\qquad{\rm and}\qquad\qquad(\D, \F^{a\circ\gamma})\quad{\rm on}\quad L^2(\R^n, \Gamma[a\circ\gamma])
\]
are unitarily equivalent, closed quadratic forms. They share, in particular, the same spectrum.
\end{thm}
\begin{proof}
i) Since $\Gamma[a]$, as an energy measure of the regular Dirichlet space $(\D,H^{1,2}(\R^n))$ on $L^2(\R^n,dx)$, is a Radon measure charging no $\D$-polar sets (see \cite[Chapter 5]{CF}) and, by assumption, it has full topological support, the quadratic form $(\D,\F^a)$ on $L^2(\R^n, \Gamma[a])$ is just an instance of speed measure change and it is thus a regular Dirichlet form by \cite[Corollary 5.2.10]{CF} and subsequent comments.\\
ii) Since by Theorem \ref{conformalactiononmultipliers}, $a\circ\gamma$ is a multiplier, if we prove that ${\rm supp}(\Gamma[a\circ\gamma])=\R^n$,
then item i) applies so that $(\D, \F^{a\circ\gamma})$ is a regular Dirichlet form on $L^2(\R^n, \Gamma[a\circ\gamma])$. The stated unitary equivalence is then a consequence of Theorem \ref{energyinvariance} and Theorem \ref{L2unitary}.\\
To prove that ${\rm supp}(\Gamma[a\circ\gamma])=\R^n$ notice that $\Gamma[a\circ\gamma]=\gamma^{-1}(|J_{\gamma^{-1}}|^{2/r}\cdot \Gamma[a])$, by Theorem \ref{L2unitary}. Since any M\"obius transformation $\gamma$ is a homeomorphism outside a suitable finite set and $J_{\gamma^{-1}}\neq 0$ $m$-a.e., the property then follows.
\end{proof}
The above result can be localized to the Dirichlet integral with absorbing boundary conditions on Euclidean domains $U\subset \R^n$. Recall that if $\F:=H^{1,2}(\R^n)$ then $\F^a:=\F_e\cap L^2(\R^n,\Gamma[a])=H^{1,2}_e(\R^n)\cap L^2(\R^n,\Gamma[a])$.
\begin{cor}
Let $a\in\M(H^{1,2}_e(\R^n))$ be a multiplier of full support ${\rm supp}(\Gamma[a])=\R^n$ and $U\subseteq\R^n$ a Euclidean domain. Then setting
\[
(\F^a)_U:=\{b\in \F^a: b=0,\,\,\Gamma[a]-{\rm a.e.\,\,on\,\,}U^c\}
\]
we have
\vskip0.1truecm\noindent
\item i) $(\D,(\F^a)_U)$ is a regular Dirichlet form on $L^2(U, \Gamma[a])$;
\vskip0.1truecm\noindent
\item ii) for a M\"obius map $\gamma:U\to\R^n$, setting $V:=\gamma(U)$, the Dirichlet forms
\[
(\D,(\F^a)_U)\quad{\rm on}\quad L^2(U, \Gamma[a])\qquad\qquad{\rm and}\qquad\qquad(\D, (\F^{a\circ\gamma})_V)\quad{\rm on}\quad L^2(V, \Gamma[a\circ\gamma])
\]
are unitarily equivalent, closed quadratic forms. They share, in particular, the same spectrum.
\end{cor}
\begin{proof}
i) The quadratic form $(\D,(\F^a)_U)$ on $L^2(U, \Gamma[a])$ is just the part of the regular Dirichlet form $(\D,\F^a)$ on $L^2(\R^n, \Gamma[a])$ (according to \cite[(3.3.1) page 108]{CF}) constructed in Theorem \ref{global} i).\\
ii) Since $J_{\gamma^{-1}}\neq 0$ on $U$, by Theorem \ref{L2unitary}, the map $\gamma_r^*:L^2(\R^n,\Gamma[a\circ\gamma])\to L^2(\R^n,\Gamma[a])$ maps unitarely the subspace $L^2(V,\Gamma[a\circ\gamma])$ onto the subspace $L^2(U,\Gamma[a])$ and maps unitarely the subspace $(\F^a)_V\subset \F^a$ onto the subspace $(\F^a)_U\subset \F^a$.
\end{proof}
\begin{rem}
i) Setting $\F:=H^{1,2}(\R^n)$, the part of $\F$ on $U$, defined as $\F_U:=\{b\in \F:b=0,\,\,dx-{\rm a.e.}\,\,{\rm on}\,\, U^c\}$, coincides with $\F_U = H^{1,2}_0(U)$.\\
ii) By \cite[Thm 3.4.9]{CF}, the extended space $(\F_U)_e$ coincides with the part $(\F_e)_U$ of $\F_e$ on $U$
\[
(\F_U)_e=\{b\in \F_e: b=0,\,\,dx-a.e.\,\,{\rm on}\,\, U^c\}
\]
where $\F_e=H^{1,2}_e(\R^n)\simeq{\overset{\circ}{{\rm BL}}(\R^n)}$, by \cite[Theorem 2.2.12]{CF}.\\
iii) The assumption ${\rm supp}(\Gamma[a])=\R^n$ allows to write $\F^a_U=\{b\in \F^a: b=0,\,\,dx-{\rm a.e.\,\,on\,\,}U^c\}$.\\
iv) Starting from the Dirichlet space $\F:=H^{1,2}(\R^n)$, the operations of taking the part on a domain $U$ and changing the speed measure from the Lebesgue one $dx$ to the energy measure $\Gamma[a]$ can be exchanged $(\F^a)_U=(\F_U)^a$ provided ${\rm supp}(\Gamma [a])=\R^n$. In fact
\[
\begin{split}
(\F^a)_U&:=\{b\in\F^a: b=0,\,\,\Gamma[a]-{\rm a.e.\,\,on\,\,}U^c\} \\
&:=\{b\in\F_e\cap L^2(\R^n,\Gamma[a]): b=0,\,\,\Gamma[a]-{\rm a.e.\,\,on\,\,}U^c\} \\
&:=\{b\in\F_e: b=0,\,\,\Gamma[a]-{\rm a.e.\,\,on\,\,}U^c\}\cap L^2(\R^n,\Gamma[a]) \\
&:=\{b\in\F_e: b=0,\,\, dx-{\rm a.e.\,\,on\,\,}U^c\}\cap L^2(\R^n,\Gamma[a]) \\
&:=(\F_e)_U\cap L^2(\R^n,\Gamma[a]) \\
&:=(\F_U)_e\cap L^2(\R^n,\Gamma[a]) \\
&=(\F_U)^a\, .
\end{split}
\]
\end{rem}
We now proceed to localize the above result to deal with cases where the energy measure $\Gamma[a]=|\nabla a|^2\cdot dx$ of a multiplier  $a\in\M(H^{1,2}_e(\R^n))$ fails to have full support $F^*_a:={\rm supp}(\Gamma[a])\neq \R^n$ ($F^*_a$ is the smallest closed set whose complement has vanishing $\Gamma[a]$ measure, hence it is the smallest closed on the complement of which $a$ is constant $dx$-a.e.).\\
To deal with this general situation we have to appeal to the notion of {\it trace of a Dirichlet form}. For reader's convenience, we now briefly recall this construction specialized to the present context and refer to \cite[Chapter 5.2 pages 176-177]{CF} for the general presentation.
\vskip0.2truecm\noindent
Any multiplier $a\in\M(H^{1,2}_e(\R^n))$ belongs to $H^{1,2}_{\rm loc}(\R^n)$ so that $\Gamma[a]:=|\nabla a|^2\cdot dx$ is a Radon measure charging no $\D$-polar sets (see \cite[Definition 2.3.13]{CF} and subsequent comments). This allows to construct the trace $(\check\D, \check\F^a)$ of the Dirichlet integral $(\D,H^{1,2}(\R^n))$ on $L^2(F^*_a,\Gamma[a])$ as follows:
\vskip0.2truecm\noindent
The subspace $H^{1,2}_{e,0}(\R^n\setminus F^*_a):=\{b\in H^{1,2}_e(\R^n): b=0\,\,{\rm q.e.}\,\,{\rm on}\,\,F^*_a\}$ is a closed subspace of the extended Dirichlet space $H^{1,2}_e(\R^n)$ which, together with its orthogonal complement $\mathcal{H}_{F^*_a}$, determines the orthogonal splitting $H^{1,2}_e(\R^n)=H^{1,2}_{e,0}(\R^n\setminus F^*_a)\oplus \mathcal{H}_{F^*_a}$. The orthogonal projection ${\bf H}_{F^*_a}:H^{1,2}_e(\R^n)\to H^{1,2}_e(\R^n)$ onto $\mathcal{H}_{F^*_a}$ assigns to a function $b\in H^{1,2}_e(\R^n)$ the unique function ${\bf H}_{F^*_a}(b)\in H^{1,2}_e(\R^n)$ which coincides $dx$-a.e. with $b$ on the closed set $F^*_a$ and is harmonic on the open complement
\[
\D(h|{\bf H}_{F^*_a}(b))=0\qquad h\in H^{1,2}_{e,0}(\R^n\setminus F^*_a)\, .
\]
The trace $({\check \D},{\check \F^a})$ of $(\D,H^{1,2}(\R^n))$ on $L^2(F^*_a,\Gamma[a])$ is the quadratic form defined as follows
\[
\check{\F^a}:=\{b|_{F^*_a}:b\in H^{1,2}_e(\R^n)\}\cap L^2(F_a,\Gamma[a])\qquad \check{\D}[b|_{F^*_a}]:=\D[{\bf H}_{F^*_a}(b)]\qquad b\in H^{1,2}_e(\R^n)\, .
\]
Since $a$ is a multiplier  $H^{1,2}_e(\R^n)\subset L^2(F^*_a,\Gamma[a])$, the domain of the trace form reduces to
\[
{\check\F^a}=\{b|_{F^*_a}:b\in H^{1,2}_e(\R^n)\}\, .
\]
This shows that  ${\check\F^a}$ depends on $a$ only through the support $F^*_a$ of its energy measure $\Gamma[a]$.
\begin{thm}
Let $a\in\M(H^{1,2}_e(\R^n))$ be a multiplier of the extended Dirichlet space. Then
\vskip0.1truecm\noindent
\item i) the trace Dirichlet integral $(\check \D, \check\F^a)$ is a regular Dirichlet form on $L^2(F^*_a, \Gamma[a])$;
\vskip0.1truecm\noindent
\item ii) for any fixed M\"obius transformation $\gamma\in G(\R^n)$, the trace Dirichlet forms
\[
(\check\D, \check\F^a)  \quad on \quad L^2(F^*_a, \Gamma[a])\qquad and \qquad
(\check\D, \check\F_{a\circ\gamma}) \quad on \quad L^2(F^*_{a\circ\gamma}, \Gamma[a\circ\gamma])
\]
are unitarily equivalent as closed quadratic forms. In particular, they share the same spectrum.
\end{thm}
\begin{proof}
i)  This follows by a direct application of \cite[Theorem 5.2.13]{CF} to the present setting.
ii) Since, by Theorem \ref{conformalactiononmultipliers}, the function $a\circ\gamma$ is a multiplier too, applying the result in i) we get that $(\check\D, \check\F_{a\circ\gamma})$  is a regular Dirichlet form on $L^2(F^*_{a\circ\gamma}, \Gamma[a\circ\gamma])$.\\
To compare the two forms we start showing that the unitary map $\gamma_r^* :H^{1,2}_e(\R^n)\to H^{1,2}_e(\R^n)$ restricts to a unitary map from the subspace $H^{1,2}_{e,0}(\R^n\setminus F^*_{a\circ\gamma})$ onto the subspace $H^{1,2}_{e,0}(\R^n\setminus F^*_{a})$.\\
Notice first that, by the regularity of the Dirichlet form, functions in an extended Dirichlet space
coincide $dx$-a.e. if and only if they coincide q.e. and then
$H^{1,2}_{e,0}(\R^n\setminus F^*_a):=\{b\in H^{1,2}_e(\R^n): b=0,\,\,\,dx-{\rm a.e.}\,\,{\rm on}\,\,F^*_a\}$ as well as $H^{1,2}_{e,0}(\R^n\setminus F^*_{a\circ\gamma}):=\{b\in H^{1,2}_e(\R^n): b=0,\,\,\,dx-{\rm a.e.}\,\,{\rm on}\,\,F^*_{a\circ\gamma}\}$.\\
To circumvent the difficulty that a general a M\"obius transformation $\gamma\in G(\R^n))$ is not an homeomorphism of $\R^n$, we use the stereographic projection $S:\R^n\to \mathbb{S}^n$ to transfer support considerations on the unit sphere $\mathbb{S}^n$ with the advantage that there the transformations $\widetilde\gamma:=S\circ\gamma\circ S^{-1}$ are homeomorphisms (in fact diffeomorphisms) belonging to the conformal group $G(\mathbb{S}^n)$.\\
Setting $\nu:=\Gamma[a]$, $\mu:=\Gamma[a\circ\gamma]$, $h:=|J_{\gamma^{-1}}|^{2/r}$ and $\widetilde\nu:=S(\nu)$, $\widetilde\mu:=S(\mu)$, $\widetilde h:=h\circ S^{-1}$, by Proposition \ref{L2unitary}, we have the relations $\gamma(\mu)=h\cdot\nu$ and also $\widetilde\gamma(\widetilde\mu)=\widetilde h\cdot\widetilde\nu$. Since $h\neq 0$ and $\widetilde h\neq 0$ a.e. on $\R^n$ and $\mathbb{S}^n$, respectively and since $\widetilde\gamma$ is an homeomorphism, we have
\[
{\rm supp}(\widetilde\nu)={\rm supp}(\widetilde h\cdot\widetilde\nu)={\rm supp}(\widetilde\gamma(\widetilde\mu))=\widetilde\gamma({\rm supp}(\widetilde\mu))\, .
\]
For $b\in H^{1,2}_e(\R^n)$, setting $\widetilde b:=b\circ S^{-1}$, the defining relation $\gamma_r^*(b):=|J_{\gamma^{-1}}|^{1/r}\cdot b$ is transformed into $\widetilde{\gamma_r^*(b)}=|J_{\gamma^{-1}}\circ S^{-1}|^{1/r}\cdot (\widetilde b\circ \widetilde\gamma^{-1})$. Hence, $\widetilde b=0$ a.e. on ${\rm supp}(\widetilde\mu)$ if and only if $\widetilde{\gamma_r^*(b)}=0$ a.e. on $\widetilde\gamma({\rm supp}(\widetilde\mu))={\rm supp}(\widetilde\nu)$. On the other hand, $\widetilde b=0$ a.e. on ${\rm supp}(\widetilde\mu)$ if and only if $b=0$ a.e. on $S^{-1}({\rm supp}(\widetilde\mu))=S^{-1}({\rm supp}(S(\mu)))={\rm supp}(\mu)$ and, similarly,
$\widetilde{\gamma_r^*(b)}=0$ a.e. on ${\rm supp}(\widetilde\nu)$ if and only if $\gamma_r^*(b)=0$ a.e. on ${\rm supp}(\nu)$. This proves that $\gamma_r^*$ maps unitarily the subspace $H^{1,2}_{e,0}(\R^n\setminus F^*_{a\circ\gamma})$ onto the subspace $H^{1,2}_{e,0}(\R^n\setminus F^*_{a})$, as required, and also that the same is true for their orthogonal complements $\mathcal{H}_{F^*_{a\circ\gamma}}$ and $\mathcal{H}_{F^*_a}$. This is equivalent to say that the unitary map $\gamma_r^*$ on the Hilbert space $H^{1,2}_e(\R^n)$ intertwines the orthogonal projections onto the subspaces $\mathcal{H}_{F^*_{a\circ\gamma}}$ and $\mathcal{H}_{F^*_a}$
\[
\gamma_r^*({\bf H}_{F^*_{a\circ\gamma}}(b))={\bf H}_{F^*_a}(\gamma_r^*(b))\qquad b\in H^{1,2}_e(\R^n)
\]
so that
\[
{\check \D}[\gamma_r^*(b)]=\D[{\bf H}_{F^*_a}(\gamma_r^*(b))]=\D[\gamma_r^*({\bf H}_{F^*_{a\circ\gamma}}(b))]=\D[{\bf H}_{F^*_{a\circ\gamma}}(b)]={\check \D}[b]\qquad b\in {\check\F}_{a\circ\gamma}\, .
\]
Since $\gamma_r^*$ is unitary between $L^2(F^*_{a\circ\gamma},\Gamma[a\circ\gamma])$ and $L^2(F^*_a,\Gamma[a])$, the proof is completed.
\end{proof}

\section{Conformal volume and fundamental tone of multipliers}
To detect the distortion of maps on $\mathbb{R}^n$, we use the sharp upper bounds on the {\it fundamental tone or first nonzero eigenvalue} of a Riemannian manifold $(M,g_M)$ with Neumann conditions on the boundary $\partial M$, in case $\partial M$ is not empty, due to Li-Yau [LY] for surfaces, El Soufi-Ilias [EI] and Colbois-El Soufi-Savo [CES] for more general weighted manifolds.
\par\noindent
These bounds involve the notion of {\it conformal volume} of a compact Riemannian manifold, a notion introduced in [LY] in their approach to the Willmore conjecture, which we now recall. Since now we will deal with Euclidean spaces with dimension $n\ge 2$.
\vskip0.2truecm\noindent
Let $(M,g_M)$ be a $m$-dimensional compact Riemannian manifold (possibly with boundary) and denote by $V(M,g_M)$ its volume and by $[g_M]$ the conformal class of the metric. By a well known theorem of J. Nash, it can be embedded isometrically into a Euclidean space $\mathbb{R}^n$ for a suitable $n\ge 1$ and then, by stereographic projection, there exists a conformal embedding of $(M,g_M)$ into the unit sphere $\mathbb{S}^n\subset \mathbb{R}^{n+1}$, endowed with its canonical metric $g_{\mathbb{S}^n}$. For any $n\ge 1$, denote by ${\rm Conf\,}((M,g_M),(\mathbb{S}^n,g_{\mathbb{S}^n}))$ the space of all conformal smooth transformations from $(M,g_M)$ to $(\mathbb{S}^n,g_{\mathbb{S}^n})$.
\par\noindent
For any conformal embedding $\phi\in {\rm Conf\,}((M,g_M),(\mathbb{S}^n,g_{\mathbb{S}^n}))$, consider the volume $V(M,\phi^*g_{\mathbb{S}^n})$ of $M$ with respect to the volume form associated to the pull-back of the canonical metric of the sphere. Later, consider the conformal group
$G(\mathbb{S}^n)={\rm Conf\,}(\mathbb{S}^n,g_{\mathbb{S}^n})$ of the sphere and set
\[
V_c^n (M,\phi) :=\sup \{V(M,(\psi\circ\phi)^*g_{\mathbb{S}^n}): \psi\in {\rm Conf\,}(\mathbb{S}^n,g_{\mathbb{S}^n})   \}\qquad n\ge 1\, .
\]
Then define the (decreasing) sequence of the {\it n-conformal volumes} by
\[
V_c^n(M,[g_M]):=\inf \{V_c^n (M,\phi):\phi\in {\rm Conf\,}((M,g_M),(\mathbb{S}^n,g_{\mathbb{S}^n}))\}\qquad n\ge 1
\]
and the {\it conformal volume} of $(M,g_M)$ by
\[
V_c(M,[g_M]):=\inf \{V_c^n(M,[g_M]):n\ge 1\}\, .
\]
The notation $V_c(M,[g_M])$ is aimed to remind that the conformal volume, by definition, depends on the conformal class of the metric only. Among its basic properties we recall the following:
\begin{enumerate}
\item ({\it conformal invariance}) conformally equivalent manifolds $(M,[g_M])\simeq^c (N,[g_N])$ share the same conformal volume as follows directly from its definition;
\item ({\it normalization}) $V_c((\mathbb{S}^n,[g_{\mathbb{S}^n})])=V(\mathbb{S}^n,g_{\mathbb{S}^n})$: see [LY Fact 2 page 272];
\item ({\it lowest bound property}) if $M$ is $n$-dimensional then, again by [LY Fact 2 page 272],
\[
V(\mathbb{S}^n ,g_{\mathbb{S}^n}) \le V_c (M,[g_M])\, ;
\]
\item ({\it monotonicity}) if $N\subseteq M$ is a subdomain then, as observed in [LY Fact 5 page 273],
\[
V_c(N,[g_M|_N])\le V_c(M,[g_M])\, .
\]
\end{enumerate}
Particularly important for us will be the property
\begin{lem}\label{conformallemma}
If $M\subset \mathbb{R}^n$ is a relatively compact subdomain, then
\[
V_c(M,[g_{\mathbb{R}^n}])= V(\mathbb{S}^n,g_{\mathbb{S}^n}).
\]
\end{lem}
\begin{proof}
By relative compactness and the fact that it has non empty interior, $M$ lies in between two Euclidean balls which share the same conformal volume $V_c(\mathbb{B}^n,[g_{\mathbb{B}^n}])$, as they are both conformally equivalent to the unit ball $\mathbb{B}^n$. Thus by the lowest bound, monotonicity and normalization properties, we have $V(\mathbb{S}^n,g_{\mathbb{S}^n})\le V_c(M,[g_{\mathbb{R}^n}])\le V_c(\mathbb{B}^n,[g_{\mathbb{B}^n}])$. On the other hand, by stereographic projection, $\mathbb{B}^n$ is conformally equivalent to an hemisphere so that, by monotonicity and normalization, we have $V_c(\mathbb{B}^n,[g_{\mathbb{B}^n}])\le V(\mathbb{S}^n,g_{\mathbb{S}^n})$.
\end{proof}
Let us consider a compact, $n$-dimensional Riemannian manifold $(M,g_M)$ (possibly with non-empty boundary $\partial M$) and its Dirichlet integral on $L^2(M,m_{g_M})$
\[
\D_{g_M}[b]:=\int_M |\nabla b|^2 dm_{g_M}\qquad b\in H^1(M,g_M)
\]
defined on the Sobolev space $H^1(M,g_M)$. A first interplay between conformal volume and spectral properties was discovered by Li-Yau in [LY Theorem 1] for surfaces and later generalized by El Soufi-Ilias in [EI Theorem 2.2] to compact manifolds of arbitrary dimension: if
$\mu_1(M,m_{g_M})>0$ is the first non-zero eigenvalue of the Dirichlet integral above, then
\begin{equation}\label{ESIbound}
\mu_1(M,m_{g_M})\cdot V(M,g_M)^{2/n}\le n\cdot V_c(M,[g_M])^{2/n}\, .
\end{equation}
This result can be compared to a classical one, due to G. Szeg\"o for planar domains and to H.F. Weinberger in greater dimension, by which, among the class of finite volume domains $\O\subset \mathbb{R}^n$ with smooth boundary, balls sharing the same volume of $\Omega$, maximize the first non-zero eigenvalue of the Laplacian: this can be restated as
\[
\mu_1(\O ,dx)\cdot V(\O)^{2/n}\le \mu_1(\mathbb{B}^n,g_{\mathbb{B}^n})\cdot V(\mathbb{B}^n)^{2/n}\, .
\]
The version of these type of results that we will need is a more recent one, due to Colbois-El Soufi-Savo [CES Theorem 3.1], which allows to consider weights both in the Dirichlet integral as well as in the reference measure. For simplicity, we just recall the restricted version we need in which no weight appear in the Dirichlet integral.
\begin{thm}{(Colbois-El Soufi-Savo [CES Theorem 3.1])}\label{CES}
Let $(M,g_M)$ be a compact, $n$-Riemannian manifold (possibly with non-empty boundary $\partial M$) and let $m_{g_M}$ be its Riemannian measure. Fix a nonnegative, non-atomic, finite Radon measure $\nu$ on $M$ charging  no $\D_{g_M}$-polar sets, so that the Dirichlet integral
\begin{equation}\label{Dirichletform}
\D_{g_M}[b]:=\int_M |\nabla b|^2 dm_{g_M}\qquad b\in H^1(M,g_M)\, ,
\end{equation}
defined on the Sobolev space $H^{1,2}(M,g_M)$, is closable on $L^2(M,\nu)$. Then, denoting by
\begin{equation}
\mu_1(M,\nu):=\inf\{ \D_{g_m}[b]: \|b\|_{L^2(M,\nu)}=1\, ,b\in H^{1,2}(M), b\perp 1\}\,
\end{equation}
the fundamental tone of the closure of the Dirichlet integral $(\D_{g_M},H^{1,2}(M))$ on $L^2(M,\nu)$, the following upper bound holds true
\begin{equation}\label{CESbound}
\mu_1(M,\nu)\le n\cdot \Bigl(\frac{V_c(M,[g_M])}{V(M,g_M)}\Bigr)^{2/n}\cdot \frac{V(M,g_M)}{\nu(M)}\, .
\end{equation}
\end{thm}
\noindent
In the following, when the manifold $M$ is a Euclidean ball $B\subset\R^n$ and the measure $\nu$ coincides with the measure ${\bf 1}_B\cdot dx$, we will adopt the notation $\mu_1(B)$ in place of $\mu_1({\bf 1}_B\cdot dx)$  for the first nonzero eigenvalue of the Dirichlet integral on $B$.

\section{Bounded distortion of fundamental tones, bounded distortion of maps}

In the forthcoming sections, Euclidean spaces $\mathbb{R}^n$ will have dimension $n\ge 2$.
\vskip0.2truecm
The main result of this section deals with the distortion properties of maps $\gamma$ of a Euclidean domain $U\subseteq\mathbb{R}^n$ into another one $\gamma(U)\subseteq\mathbb{R}^n$, which transform, by composition, the finite energy multipliers algebra $\F\M(H^{1,2}(\gamma(U)))$ into the finite energy multipliers algebra $\F\M(H^{1,2}(U))$.
Before stating the main result, we need to prove the following
\begin{lem}\label{integrability}
Let $\gamma:U \rightarrow\mathbb{R}^n$ be a local homeomorphism  such that $a\circ\gamma\in \M(H^{1,2}(U))\cap H^{1,2}(U)$ is a finite energy multiplier for any finite energy multiplier $a\in \M(H^{1,2}(\gamma(U))\cap H^{1,2}(\gamma(U))$.
\par\noindent
Then $\gamma\in H^{1,2}_{\rm loc}(U,\mathbb{R}^n)$ and the Jacobian determinant is locally integrable $J_\gamma\in L^1_{\rm loc}(U)$.
\end{lem}
\begin{proof} For each $k=1,\cdots ,n$ and each bounded open subset $U^\prime \subset U$, since the coordinate function $p_k(y):=y_k$ defined for $y=(y_1,\cdots ,y_n)\in\mathbb{R}^n$ is smooth, there exist a function $a_k\in C^\infty_c(\gamma(U))\subset \M(H^{1,2}(\gamma(U))\cap H^{1,2}(\gamma(U))$ such that $a_k(y)=p_k(y)=y_k$ for $y\in \gamma(U^\prime)$. By hypothesis, we have $a_k\circ\gamma\in \M(H^{1,2}(U)\cap H^{1,2}(U)$ so that $(a_k\circ\gamma)b\in H^{1,2}(U)$ for any $b\in H^{1,2}(U)$. Choosing $b\in C^1_c(U)\subset H^{1,2}(U)$ such that $b(x)=1$ for $x\in U^\prime$  and denoting by $\gamma_k:=p_k\circ\gamma$ the $k$-component of $\gamma$, we have $\gamma_k(x)=(p_k\circ\gamma)(x)=a_k (\gamma(x))b(x)$ for $x\in U^\prime$. Since the bounded open set $U^\prime\subset U$ can be chosen arbitrarily, we proved that $\gamma_k \in H^{1,2}_{\rm loc}(U)$ for any $k=1,\cdots,n$. Since $H^{1,2}_{\rm loc}(U)\subset H^{1,1}_{\rm loc}(U)$, we have $J_\gamma\in L^1_{\rm loc}(U)$  by Corollary \ref{J2}.
\end{proof}
The following are the main results of this section. They ascribe the bounded distortion of a map to the bounded distortion of the fundamental tones of multipliers of $\D$.
\begin{thm} \label{maintheorem}
Let $U\subseteq \R^n$ be a domain and $\gamma:U \rightarrow\mathbb{R}^n$ a local homeomorphism such that
\vskip0.1truecm\noindent
1) for any relatively compact domain $A\subseteq U$, $a\circ\gamma\in \F\M(H^{1,2}(A))$ is a finite energy multiplier for any finite energy multiplier
$a\in \F\M(H^{1,2}(\gamma(A))$
\vskip0.1truecm\noindent
2) $J_\gamma(x)\neq 0$, $dx$-a.e. $x\in U$ and $J_\gamma$ has constant sign in $U$.
\vskip0.2truecm\noindent Then
\vskip0.2truecm\noindent
i) if for some $K> 0$, any relatively compact domains $A\subseteq U$ and any nowhere constant, finite energy multiplier $a\in \F\M(H^{1,2}(\gamma(A)))$
one has
\begin{equation}\label{boundedness1.0}
\mu_1(\gamma(A), a)\le K^2\cdot\mu_1(A, a\circ\gamma)\, ,
\end{equation}
then $\gamma$ is a bounded distortion map with distortion coefficient bounded by 
\[
K\le K(\gamma)\le K \sqrt{\frac{n}{n-1}}\Bigl(\frac{V(\mathbb{S}^n)}{ V(\mathbb{B}^n)}\Bigr)^{1/n}\, .
\]
ii) In particular, if $\gamma$ is a homeomorphism and for $K> 0$, any relatively compact domain $A\subseteq U$ and any nowhere constant, finite energy multiplier $a\in \F\M(H^{1,2}(\gamma(A)))$
one has
\begin{equation}\label{boundedness1.1}
K^{-2}\cdot \mu_1(\gamma(A),a)\le\mu_1(A,a\circ\gamma)\le K^2\cdot\mu_1(\gamma(A),a)\, ,
\end{equation}
then $\gamma$ is quasiconformal.
\end{thm}
\begin{proof}
Notice first that by assumption 1) and Lemma \ref{integrability}, and by assumption 2), the requirements i) and ii) in Definition \ref{BD} are satisfied. Notice also that if $a\in\F\M(H^{1,2}(\gamma(A)))$ is nowhere constant, then ${\rm supp}(\Gamma[a])=\gamma(A)$ and we can consider, by Proposition \ref{DF}, the Dirichlet form $(\D,\F^a)$ on $L^2(\gamma(A),\Gamma[a])$ for any relatively compact domains $A\subseteq U$. Moreover, since $\gamma$ is a local homeomorphism, then $a\circ\gamma\in\F\M(H^{1,2}(A))$ is nowhere constant too so that ${\rm supp}(\Gamma[a\circ\gamma])=A$ and we can consider the Dirichlet form $(\D,\F^{a\circ\gamma})$ on $L^2(A,\Gamma[a])$ for any relatively compact domains $A\subseteq U$.\\
Restricting the local homeomorphism $\gamma$ to the sets of a suitable open cover of $U$, we may reduce the proof to the case of a global homeomorphism. {\it Since now on we will consider $\gamma$ to be an homeomorphism between the domains $U$ and $\gamma(U)$}.
\vskip0.2truecm\noindent
{\bf i}) Let $B\subseteq \gamma(U)\subset\R^n$ be bounded domain having the extension property and continuous boundary (see Section 2.3) and consider the bounded domain $A:=\gamma^{-1}(B)\subset\R^n$. By the definition of the fundamentals tones, we have
\begin{equation}\label{eigenvalue1}
\mu_1(B, a)=\inf_{b\in\F^a ,\D[b]>0,b\perp 1}\frac{\mathcal{D}[b]}{\|b\|^2_{L^2(B,\Gamma[a])}}\, ,\qquad
\mu_1(A,a\circ\gamma)=\inf_{b\in\F^{a\circ\gamma} ,\D[b]>0, b\perp 1}\frac{\mathcal{D}[b]}{\|b\|^2_{L^2(A,\Gamma[a\circ\gamma])}}
\end{equation}
so that assumption \eqref{boundedness1.0} can then be rewritten as
\begin{equation}\label{boundedness1.0.1}
\sup_{b\in\F^{a\circ\gamma} ,\D[b]=1, b\perp 1}\int_{A} |b(x)|^2|\nabla (a\circ\gamma)(x)|^2\, dx  \le K^2\cdot \sup_{b\in\F^{a} ,\D[b]=1, b\perp 1}\int_{B} |b(y)|^2|\nabla a(y)|^2\, dy\, .
\end{equation}
Choose $a\in C^1(\gamma(U))$ such that $|\nabla a|=1$ on $\gamma(U)$. Since $B$ has the extension property, the restriction of $a$ to $B$ is a finite energy multiplier of $H^{1,2}(B)$ and, by the assumption 1), $a\circ\gamma$ is a finite energy multiplier of $H^{1,2}(A)$. From \eqref{boundedness1.0.1} we have
\begin{equation}\label{boundedness1.0.2}
\sup_{b\in\F^{a\circ\gamma} ,\D[b]=1, b\perp 1}\int_{A} |b(x)|^2|\nabla (a\circ\gamma)(x)|^2\, dx  \le K^2\cdot \sup_{b\in\F^{a} ,\D[b]=1, b\perp 1}\int_{B} |b(y)|^2\, dy\, .
\end{equation}
Since, by assumption 1), $a\circ\gamma$ is a finite energy multiplier, the Radon measure $\nu:=\Gamma [a\circ\gamma]$ on $A$ is finite and we can apply \cite[Theorem 3.1]{CES} (or Theorem \ref{CES} above) to the closure of the quadratic form $(\D,H^{1,2}(A))$ in \eqref{Dirichletform}, which is nothing but the Dirichlet space $(\D,\F^{a\circ\gamma})$ on $L^2(A,\Gamma[a\circ\gamma])$. Noticing that
\[
\nu(A)=\int_A|\nabla(a\circ\gamma)(x)|^2\, dx\, ,
\]
the Colbois-El Soufi-Savo bound \eqref{CESbound} allows to estimate from below the l.h.s. of \eqref{boundedness1.0.2} to get
\begin{equation}\label{lowerbound}
\frac{1}{n}\Bigl(\frac{V(A)}{V_c(A)}\Bigr)^{2/n}
\frac{1}{V(A)}\int_{A}|\nabla(a\circ\gamma)(x)|^2\, dx \le K^2\cdot \sup_{b\in\F^{a} ,\D[b]=1, b\perp 1} \int_{B} |b(y)|^2\, dy\, .
\end{equation}
\noindent
Since we assumed that $|\nabla a|=1$ on $B$, we have $\Gamma[a]=dx$ and $\F^a=H^{1,2}(B)$. The Dirichlet form $(\D,\F^a)$ on $L^2(B,\Gamma[a])=L^2(B,dx)$ is just the Dirichlet form of the Brownian motion in $B\cup \partial B$ reflected at $\partial B$ (see \cite[Section 2.2.4 ]{CF}). In particular, the fundamental tone $\mu_1(B, a)$ is just  the first nonzero eigenvalue $\mu_1(B)$ of the Neumann Laplacian on $B$ and we can write (\ref{lowerbound}) as
\begin{equation}\label{lowerbound.1}
\frac{1}{n}\Bigl(\frac{V(A)}{V_c(A)}\Bigr)^{2/n}
\frac{1}{V(A)}\int_{A}|\nabla(a\circ\gamma)(x)|^2\, dx \le K^2\cdot \mu_1(B)^{-1}\, .
\end{equation}
To bound above the r.h.s. of \eqref{lowerbound.1}, let us introduce the {\it effective conformal volume} $V^\prime_c(B)$ by
\begin{equation}\label{effectivevolume}
\mu_1(B)^{-1}=: \frac{1}{n}\Bigl(\frac{V(B)}{V^\prime_c(B)}\Bigr)^{2/n}\, .
\end{equation}
By the El Soufi-Ilias bound \eqref{ESIbound}, we have
\begin{equation}
\frac{1}{n}\Bigl(\frac{V(B)}{V_c(B)}\Bigr)^{2/n}\le \mu_1 (B)^{-1}=: \frac{1}{n}\Bigl(\frac{V(B)}{V^\prime_c(B)}\Bigr)^{2/n}
\end{equation}
so that the effective conformal volume does not exceed its conformal volume: $V^\prime_c (B)\le V_c(B)$.
By \eqref{effectivevolume} and \eqref{lowerbound.1} we thus have
\begin{equation}\label{lowerbound1.1}
\frac{1}{n}\Bigl(\frac{V(A)}{V_c(A)}\Bigr)^{2/n}
\frac{1}{V(A)}\int_{A}|\nabla (a\circ\gamma)(x)|^2\, dx   \le K^2\cdot \frac{1}{n}\Bigl(\frac{V(B)}{V^\prime_c(B)}\Bigr)^{2/n}
\end{equation}
which can be written as
\begin{equation}\label{lowerbound1.2}
\frac{1}{V(A)}\int_{A}|\nabla (a\circ\gamma)(x)|^2\, dx   \le K^2\cdot \Bigl(\frac{V_c(A)}{V^\prime_c(B)}\Bigr)^{2/n}\cdot
\Bigl(\frac{V(B)}{V(A)}\Bigr)^{2/n}
\end{equation}
\vskip0.2truecm\noindent
for any  bounded domain $B\subseteq \gamma(U)$ having the extension property and continuous boundary.
\vskip0.2truecm\noindent
{\bf ii}) Assume now that $B\subseteq \gamma(U)$ is a ball of radius $r>0$. By a standard result, (see [Ch Thm 4 Chapter II]), $\mu_1(B)$ depends upon the radius through the relation
\begin{equation}
\mu_1(B)=\frac{c_n^{2}}{r^2}\, ,
\end{equation}
(in particular $\mu_1(\mathbb{B}^n)=c_n^2$) where $c_n$ is the first critical point of the first order Bessel function, solution to the ordinary differential equation
\begin{equation}\label{differential}
 z^{\prime\prime}(t)+\frac{n-1}{t}z^\prime(t) +\Bigl(1-\frac{n-1}{t^2}\Bigr)z(t)=0\qquad t\ge 0
\end{equation}
subject to the initial conditions $z(t)=0\, ,z^\prime (0)=1$.
Since $V(B)=r^nV(\mathbb{B}^n)$, by \eqref{effectivevolume} we have
\[
n\Bigl(\frac{V^\prime_c(B)}{r^nV(\mathbb{B}^n)}\Bigr)^{2/n} =\frac{c_n^{2}}{r^2}
\]
so that the effective conformal volume is independent on center and radius of $B$ and given by
 \begin{equation}\label{effectivevolumeball}
V_c^\prime(B)=V_c^\prime(\mathbb{B}^n)=V(\mathbb{B}^n)\Bigl(\frac{c_n^{2}}{n}\Bigr)^{n/2}=V(\mathbb{B}^n)\Bigl(\frac{c_n}{\sqrt{n}}\Bigr)^{n}.
\end{equation}
Since, by Lemma \ref{conformallemma}, $V_c(A)=V(\mathbb{S}^n)$, inequality \eqref{lowerbound1.2} can be written as
\begin{equation}\label{lowerbound1.3}
\frac{1}{V(A)}\int_{A}|\nabla (a\circ\gamma)(x)|^2\, dx   \le K^2\cdot \Bigl(\frac{V(\mathbb{S}^n)}{ V^\prime_c(\mathbb{B}^n)}\Bigr)^{2/n}\cdot \Bigl(\frac{V(B)}{V(A)}\Bigr)^{2/n}
\end{equation}
and as
\begin{equation}\label{lowerbound1.4}
\frac{\Gamma[a\circ\gamma](A)}{V(A)} \le K^2\cdot \Bigl(\frac{V(\mathbb{S}^n)}{ V^\prime_c(\mathbb{B}^n)}\Bigr)^{2/n}\cdot \Bigl(\frac{V(B)}{V(A)}\Bigr)^{2/n}\, .
\end{equation}
Since $\Gamma[a\circ\gamma](A)=\Gamma[a\circ\gamma](\gamma^{-1}(B))=\gamma(\Gamma[a\circ\gamma])(B)$, denoting by $m$ the Lebesgue measure, we have $V(B)=m(B)$, $V(A)=m(A)=m(\gamma^{-1}(B))=\gamma(m)(B)$ so that (\ref{lowerbound1.4}) can be written
\begin{equation}\label{lowerbound1.5}
\frac{\gamma(\Gamma[a\circ\gamma])(B)}{\gamma(m)(B)} \le K^2\cdot \Bigl(\frac{V(\mathbb{S}^n)}{ V^\prime_c(\mathbb{B}^n)}\Bigr)^{2/n}\cdot \Bigl(\frac{m(B)}{\gamma(m)(B)}\Bigr)^{2/n}\, .
\end{equation}
\noindent
The arbitrariness of the ball $B\subseteq\gamma(U)$ and a double application of the Lebesgue Differentiation Theorem allow to obtain
\begin{equation}\label{lowerbound1.6}
\frac{d\gamma(\Gamma[a\circ\gamma])}{d\gamma(m)} \le K^2\cdot \Bigl(\frac{V(\mathbb{S}^n)}{ V^\prime_c(\mathbb{B}^n)}\Bigr)^{2/n}\cdot \Bigl(\frac{dm}{d\gamma(m)}\Bigr)^{2/n}\qquad \gamma(m)-a.e.\quad{\rm on}\quad \gamma(U).
\end{equation}
for any $a\in C^1(\gamma(U))$ such that $|\nabla a|=1$ on $\gamma(U)$.
Let us compute the Radon-Nikodym derivatives appearing in (\ref{lowerbound1.6}). On one hand, for any $h\in C_c(\gamma(U))$ we have
\[
\int_{\gamma(U)}h(y)\, \gamma(m)(dy)=\int_U h(\gamma(x))m(dx)=\int_{\gamma(U)}h(y)\cdot |J_\gamma(\gamma^{-1}(y))|^{-1}\, m(dy)
\]
so that $\frac{dm}{d\gamma(m)}(y)=|J_\gamma(\gamma^{-1}(y))|$, for $\gamma(m)$-a.e. $y\in \gamma(U)$. On the other hand, for any $h\in C_c(\gamma(U))$ we have
\[
\begin{split}
\int_{\gamma(U)}h(y)\gamma(\Gamma[a\circ\gamma])(dy)&=\int_U h(\gamma(x))\,\Gamma[a\circ\gamma](dx)\\
&=\int_U h(\gamma(x))\cdot |\nabla(a\circ\gamma)(x)|^2\, m(dx)\\
&=\int_U h(\gamma(x))\cdot |\nabla(a\circ\gamma)(\gamma^{-1}(\gamma(x)))|^2\, m(dx)\\
&=\int_U h(y)\cdot |\nabla(a\circ\gamma)(\gamma^{-1}(y))|^2\, \gamma(m)(dy)
\end{split}
\]
so that $\frac{d\gamma(\Gamma[a\circ\gamma])}{d\gamma(m)}=|\nabla(a\circ\gamma)(\gamma^{-1}(y))|^2$, for $\gamma(m)$-a.e. $y\in \gamma(U)$. From (\ref{lowerbound1.6}) we have
\[
|\nabla(a\circ\gamma)(\gamma^{-1}(y))|^2\le  K^2\cdot \Bigl(\frac{V(\mathbb{S}^n)}{ V^\prime_c(\mathbb{B}^n)}\Bigr)^{2/n}\cdot |J_\gamma(\gamma^{-1}(y))|^{2/n}\qquad \gamma(m)-a.e.\quad{\rm on}\quad \gamma(U)
\]
and then
\begin{equation}\label{lowerbound1.7}
|\nabla(a\circ\gamma)(x)|^2\le  K^2\cdot \Bigl(\frac{V(\mathbb{S}^n)}{ V^\prime_c(\mathbb{B}^n)}\Bigr)^{2/n}\cdot |J_\gamma(x)|^{2/n}\qquad m-a.e.\quad{\rm on}\quad U\, .
\end{equation}
For a fixed $x\in U$ such that $\gamma'(x)\in \mathbb{M}_n(\mathbb{R})$ exists, choose a multipliers $a\in C^1(\gamma(U))$ such that $\nabla a(\gamma(x))$ is a unit norm eigenvector of $\gamma^\prime(x)^t$ corresponding to its largest eigenvalue $\|\gamma^\prime(x)^t\|=\|\gamma^\prime(x)\|$. We then have
\[
|\nabla(a\circ\gamma)(x)|=|\gamma^\prime(x)^t\cdot (\nabla a (\gamma(x))) |= \|\gamma^\prime(x)\|\cdot |\nabla a (\gamma(x_0))|=\|\gamma^\prime(x)\|
\]
and by (\ref{lowerbound1.7}) 
\begin{equation}\label{distortion2}
\|\gamma^\prime(x)\|\le K\cdot \Bigl(\frac{V(\mathbb{S}^n)}{ V^\prime_c(\mathbb{B}^n)}\Bigr)^{1/n}\cdot |J_{\gamma}(x)|^{1/n}\, .
\end{equation}
Since $\gamma'(x)$ exists for m-a.e. $x\in U$, this proves that (\ref{distortion2}) holds true $m$-a.e. on $U$. This implies that even the requirement iii) in Definition \ref{BD} is satisfied so that $\gamma$ has bounded distortion with distortion coefficient
\begin{equation}\label{coeff1}
K(\gamma)= K\Bigl(\frac{V(\mathbb{S}^n)}{ V^\prime_c(\mathbb{B}^n)}\Bigr)^{1/n} =K\sqrt{\frac{n}{\mu_1(\mathbb{B}^n)}}\Bigl(\frac{V(\mathbb{S}^n)}{ V(\mathbb{B}^n)}\Bigr)^{1/n}.
\end{equation}
Since $V_c'(\mathbb{B}^n)\ge V_c(\mathbb{B}^n)=V(\mathbb{S}^n)$, we have $K(\gamma)\ge K$.
\vskip0.2truecm\noindent
To obtain a weaker but more explicit estimate on $K(\gamma)$ in terms the dimension $n$, notice that the first order Bessel function, solution of \eqref{differential}, subject to the specified boundary conditions at $t=0$, satisfies $z^{\prime\prime}(c_n)\le 0$. In fact if otherwise $z^{\prime\prime}(c_n)> 0$, as $z^\prime (c_n)=0$, we would have a local minimum in $t=c_n$. This would imply the existence of a local maximum in $[0,c_n)$. But since $z(0)=0$ and $z^\prime (0)=1$, this local maximum would lie in the interval $(0,c_n)$. As the function is differentiable, there would exists a critical point in $(0,c_n)$ in contradiction with the assumption that $t=c_n$ is the first one from the left. This proves also that $z(c_n)>0$, so that by \eqref{differential} we have
\[
\Bigl(1-\frac{n-1}{c_n^2}\Bigr)\ge 0,\qquad \mu_1(\mathbb{B}^n)=c_n^2\ge n-1
\]
and then
\begin{equation}\label{coeff2}
K\le K(\gamma)\le K \sqrt{\frac{n}{n-1}}\Bigl(\frac{V(\mathbb{S}^n)}{ V(\mathbb{B}^n)}\Bigr)^{1/n}\, .
\end{equation}

\end{proof}
\noindent
In case $\gamma$ is a continuously differentiable map, the assumptions of Theorem \ref{maintheorem} simplifies to

\begin{thm}
Let $\gamma\in C^1(U,\mathbb{R}^n)$ be locally invertible and such that $J_\gamma$ has constant sign.
\vskip0.1truecm\noindent
\noindent
i) If for some $K> 0$, any relatively compact domains $A\subseteq U$ and any nowhere constant $a\in C^1(\gamma(U))$  one has
\begin{equation}
\mu_1(\gamma(A), a)\le K^2\cdot\mu_1(A, a\circ\gamma)\, ,
\end{equation}
then $\gamma$ is a map with bounded distortion.
\vskip0.2truecm\noindent
ii) In particula, if $\gamma$ is invertible and for some $K> 0$, any relatively compact domain $A\subseteq U$ and any nowhere constant $a\in C^1(\gamma(U))$  one has
\begin{equation}
K^{-2}\cdot \mu_1(\gamma(A),a)\le\mu_1(A,a\circ\gamma)\le K^2\cdot\mu_1(\gamma(A),a)\, ,
\end{equation}
then $\gamma$ is a quasiconformal map.
\end{thm}
\begin{proof}
Notice that, since $A$ is relatively compact in $U$ and $\gamma$ is continuous, $\gamma(A)$ is relatively compact in $\gamma(U)$ and then functions in $C^1(\gamma(U))$ restrict to multipliers of $H^{1,2}(\gamma(A))$. Analogously, by assumption, $a\circ\gamma$ belongs to $C^1(U)$, so that $a\circ\gamma$ is a multiplier of $H^{1,2}(A)$. This is just what is needed to repeat the reasoning in the proof of item i) in Theorem \ref{maintheorem}.
\end{proof}

\begin{rem}
i) It is enough to verify conditions (6.23), (6.24) just for a countable base of open, relatively compact sets $A\subseteq U$ and for a corresponding countable family of multipliers $a\in C^1(\gamma(U))$.\\
i) Since $V(\mathbb{S}^n)/V(\mathbb{B}^n) \simeq \sqrt{2\pi n}$ as $n\to +\infty$, the bound $(6.22)$ on the distortion coefficient tends to became optimal as the dimension became larger: $\lim_n K(\gamma)/K=1$.
\end{rem}
\section{Conclusions}
The aim of this work has been to connect, in a Euclidean setting, the quasiconformal geometry of a domain $U$ to the spectral properties of its energy integral $(\D,H^{1,2}(U))$ on $L^2(U,dx)$. A central role has been played by the algebra of multipliers $\M(H^{1,2}(U))$ whose properties reflect the potential theory of the Dirichlet space.\\
With respect to the quasiconformal geometry, $\M(H^{1,2}(U))$ plays a role alternative to the Royden algebra $H^{1,n}(U)\cap L^\infty(U,dx)$ (see [Lew], [Mos]). Recall that quasiconformal maps between domains $\gamma:U \to V$ are characterized as those which establishes an invertible endomorphism $\alpha_\gamma (a):=a\circ \gamma^{-1}$ between the normed algebras $H^{1,n}(V)\cap L^\infty(V,dx)$ and $H^{1,n}(U)\cap L^\infty(U,dx)$.
An essential difference between the algebras $\M(H^{1,2}(U))$ and $H^{1,n}(U)\cap L^\infty(U,dx)$ is that the definition of the former does not involves, explicitly, higher order integrability of the gradient of functions, as the Sobolev space $H^{1,n}(U)$ does, nor the dimension of the Euclidean space. The multiplier algebra $\M(H^{1,2}(U))$ is intrinsic to the space $H^{1,2}(U)$ and its use clarify that the quasiconformal geometry of a Euclidean domain underlies, its energy functional only, with no reference to its volume measure. This aspect is highlighted by the results of Section 4 in which the emphasis is not on the use of the multipliers of $H^{1,2}(\R^n)$ itself but rather on the use of the multipliers of the extended space $H^{1,2}_e(\R^n)$.
An advantage of the present approach is that it can be considered on any Dirichlet space like smooth manifolds, metric spaces or self-similar fractals.

\end{document}